\documentclass[11pt]{article}
\usepackage{amsfonts}
\usepackage{graphics}
\usepackage{indentfirst}
\usepackage{color}
\usepackage{cite}
\usepackage{float}
\usepackage{caption}
\usepackage{latexsym}
\usepackage[paper=a4paper, left=1.6cm, right=1.6cm, top=1.8cm, bottom=1.6cm, headheight=5.5pt, footskip=0.8cm, footnotesep=0.8cm, centering, includefoot]{geometry}
\usepackage{amsmath}
\allowdisplaybreaks
\usepackage{amssymb}
\usepackage[colorlinks, linkcolor=red]{hyperref}
\hypersetup{urlcolor=red, citecolor=red}
\usepackage[dvips]{epsfig}
\usepackage{amscd}

\usepackage{amsthm}
\usepackage{mathrsfs}
\usepackage{verbatim}
\newtheorem{theorem}{Theorem}[section]
\newtheorem{remark}{Remark}[section]
\newtheorem{definition}{Definition}[section]
\newtheorem{lemma}{Lemma}[section]
\newtheorem{corollary}{Corollary}[section]
\newtheorem{proposition}{Proposition}[section]

\allowdisplaybreaks

\DeclareMathOperator{\loc}{loc}

\makeatletter
\@addtoreset{equation}{section}
\makeatother
\makeatletter
\@addtoreset{equation}{section}
\makeatother

\title{ Global strong solutions for non-isothermal compressible nematic liquid crystal flows under a scaling-invariant smallness condition
\thanks{Xu's research was partially supported by Postgraduate Research an Innovation Project of Southwest University (No. SWUB25035). Zhong's research was partially supported by National Natural Science Foundation of China (No. 12371227) and Fundamental Research Funds for the Central Universities (No. SWU--KU24001).}
}

\author{Lin Xu,\ Xin Zhong{\thanks{E-mail addresses: mathxu@email.swu.edu.cn (L. Xu), xzhong1014@amss.ac.cn (X. Zhong).}}
\date{}\\
\footnotesize School of Mathematics and Statistics, Southwest University, Chongqing 400715, P. R. China}
\begin{document}
\maketitle
\begin{abstract}
	We study the three-dimensional Cauchy problem for a non-isothermal compressible nematic liquid crystal system with far-field vacuum. By deriving refined energy estimates and exploiting the coupled structure of the equations, we establish the global existence and uniqueness of strong solutions, provided that the following scaling-invariant quantity is sufficiently small:
	$$
	\big(1+\bar{\rho}+\tfrac{1}{\bar{\rho}}\big)
	\big[\|\rho_{0}\|_{L^{3}}+(\bar{\rho}^{2}+\bar{\rho})\big(\|\sqrt{\rho_{0}}u_{0}\|_{L^{2}}^{2}+\|\nabla d_{0}\|_{L^{2}}^{2}\big)\big]
	\big[\|\nabla u_{0}\|_{L^{2}}^{2}+(\bar{\rho}+1)\|\sqrt{\rho_{0}}\theta_{0}\|_{L^{2}}^{2}
	+\|\nabla^{2} d_{0}\|_{L^{2}}^{2}+\|\nabla d_{0}\|_{L^{4}}^{4}\big].
	$$
	In particular, our result identifies a new scaling-invariant quantity and does not impose additional restrictions on the viscosity coefficients, which improves previous work (Commun. Math. Sci. 21 (2023), 1455--1486).
\end{abstract}




\section{Introduction}
Liquid crystals are materials whose mechanical behavior lies between that of ordinary liquids and solid crystals, with elongated molecules giving rise to a preferred local orientation. Among various phases, nematic liquid crystals are of particular interest and can be described at the macroscopic level by the continuum theory of Ericksen and Leslie \cite{E62,L68}, which couples the fluid velocity with the molecular orientation. In the static regime, this theory reduces to the classical Oseen-Frank elastic model \cite{O,F}. From the original equations in the theory of nematic liquid crystals, Lin \cite{L89} derived a simplified Ericksen-Lesile model. Since then there have been many results about mathematical analysis in this model and related models.

Under suitable simplifications, the Ericksen-Leslie framework leads to a nonlinear system coupling the compressible Navier--Stokes equations with an evolution equation for the director field. The latter can be interpreted as a transported heat flow of harmonic maps into the unit sphere (see, e.g., \cite{L95}). The resulting model is referred to as the non-isothermal compressible nematic liquid crystal (NLC) system
\begin{align}\label{1}
\begin{cases}
\rho_{t}+\operatorname{div}(\rho u)=0, \\
\rho u_{t}+\rho u\cdot\nabla u-\mu \Delta u-(\lambda+\mu)\nabla \operatorname{div} u+\nabla P=-\nabla d \cdot \Delta d, \\
c_{v}\rho\big(\theta_{t}+u\cdot \nabla \theta\big)+P\operatorname{div} u-\kappa \Delta \theta
=\mathcal{Q}(\nabla u)+\big|\Delta d+|\nabla d|^{2} d\big|^{2}, \\
d_{t}+u\cdot \nabla d=\Delta d+|\nabla d|^{2} d, \\
|d|=1.
\end{cases}
\end{align}
Here $\rho\ge 0$, $u\in\mathbb{R}^{3}$, $\theta\ge 0$, and $d\in\mathbb{S}^{2}$ denote the density, velocity, absolute temperature, and orientation field, respectively. The pressure is given by
$P=R\rho\theta$ with $R$ being a positive constant. Moreover,
\begin{align*}
\mathcal{Q}(\nabla u)=\frac{\mu}{2}\,|\nabla u+(\nabla u)^{\top}|^{2}+\lambda\,(\operatorname{div}u)^{2},
\end{align*}
where $(\nabla u)^{\top}$ denotes the transpose of $\nabla u$. The shear viscosity coefficient $\mu$ and bulk viscosity coefficient $\lambda$ satisfy the physical conditions
\begin{align*}
\mu>0,\quad 2\mu+3\lambda\ge 0.
\end{align*}
The positive constants $c_{v}$ and $\kappa$ represent the heat capacity and the heat conductivity coefficient, respectively. We consider the Cauchy problem for \eqref{1} with initial data
\begin{align}\label{2}
(\rho,u,\theta,d)\big|_{t=0}=(\rho_{0},u_{0},\theta_{0},d_{0}),\quad x\in\mathbb{R}^{3},
\end{align}
and the far-field behavior
\begin{align}\label{4}
(\rho,u,\theta,d)(x,t)\to(0,0,0,\mathbf{1}) \quad \text{as } |x|\to\infty,\ t>0,
\end{align}
where $\mathbf{1}$ is a fixed unit vector.

Temperature plays a fundamental role in the macroscopic behavior of nematic liquid crystals, as it affects both molecular orientation and its coupling with the surrounding flow. In non-isothermal models, it couples the velocity, director, and energy equations through heat conduction and viscous dissipation, and may also influence the pressure and material coefficients. For background on isothermal and non-isothermal liquid crystals, refer to \cite{FRSZ14,HP18} and references contained therein. In 2011, Feireisl--Rocca--Schimperna \cite{FRS} proposed  a non-isothermal incompressible NLC model, where the constraint $|d|=1$ is relaxed by a penalization term. Later, accounting for stretching of the director, the authors \cite{FFRS} introduced a new modeling framework and proved global-in-time existence of weak solutions with a Ginzburg--Landau type director equation. More recently, increasing attention has been devoted to the compressible non-isothermal NLC system, where the analysis becomes more involved due to density variations and their impact on the fluid mechanics. In this regards, Hieber and Pr\"uss \cite{HP16} derived a general Ericksen--Leslie model. Later on, they \cite{HP19} further established the local well-posedness in an $L^{q}$ framework and obtained global-in-time solutions near equilibria. Meanwhile, De Anna and Liu \cite{DL} proposed a thermally driven model and established the global well-posedness with small data in Besov spaces.

There are various contributions to simplified non-isothermal compressible NLC systems, including the low-Mach-number limit \cite{FL}, biaxial NLC flows \cite{LWZ}, decay rates \cite{GXX17}, blow-up criteria \cite{Z19,Z22}, and so on. For strong solutions to the Cauchy problem \eqref{1}--\eqref{4}, Liu and the second author \cite{LZ21} generalized the local well-posedness in \cite{FLN} to be a global one as long as $\|\rho_{0}\|_{L^{\infty}}+\|\nabla d_{0}\|_{L^{3}}$ is sufficiently small. Subsequently, they \cite{LZ26} obtained global strong solutions and algebraic decay rates provided that the total initial energy is small enough. Meanwhile, Shou and the second author \cite{SZsub} established global strong solutions once $\|\rho_{0}\|_{L^{1}}^{2}+\|\nabla d_{0}\|_{L^{2}}^{2}$ is properly small. It should be pointed that the smallness assumptions in \cite{SZsub,LZ21,LZ26} {\it depend on the initial data}. Recently, under a smallness condition that is \emph{independent of the initial data}, it was shown in \cite{LT23} the global existence and uniqueness of strong solutions to \eqref{1}--\eqref{4} provided that the viscosity coefficients satisfy an additional condition $2\mu>\lambda$. The smallness condition in \cite{LT23} is imposed on the quantity
\begin{align*}
N_{0}\triangleq \hat{\rho}\big[\|\rho_{0}\|_{L^{3}}+\hat{\rho}^{2}\big(\|\sqrt{\rho_{0}}u_{0}\|_{L^{2}}^{2}+\|\nabla d_{0}\|_{L^{2}}^{2}\big)\big]
\big[\|\nabla u_{0}\|_{L^{2}}^{2}+\hat{\rho}\big(\|\sqrt{\rho_{0}}E_{0}\|_{L^{2}}^{2}
+\|\nabla^{2}d_{0}\|_{L^{2}}^{2}\big)\big],
\end{align*}
where $E_{0}=\frac{|u_{0}|^{2}}{2}+c_{v}\theta_{0}$ and $\hat{\rho}=\|\rho_{0}\|_{L^{\infty}}+1$. Actually,
the quantity $N_{0}$ is scaling-invariant under the parabolic scaling
\begin{align}\label{sca}
\rho_{k}(x,t)\triangleq \rho(kx,k^{2}t),\quad
u_{k}(x,t)\triangleq k u(kx,k^{2}t),\quad
\theta_{k}(x,t)\triangleq k^{2}\theta(kx,k^{2}t),\quad
d_{k}(x,t)\triangleq d(kx,k^{2}t),
\end{align}
for any $k\neq 0$. More precisely, if $N_{0}\le \delta_{0}$ holds for the initial data $(\rho_{0},u_{0},\theta_{0},d_{0})$, then the rescaled initial data $(\rho_{k}(x,0),u_{k}(x,0),\theta_{k}(x,0),d_{k}(x,0))$ defined by \eqref{sca} also satisfy $N_{0}\le \delta_{0}$. From a physical viewpoint, the viscosity restriction $2\mu>\lambda$ is not natural, since for many fluids the bulk viscosity $\lambda$ can be hundreds or even thousands of times larger than the shear viscosity $\mu$ (see, e.g., \cite{C11,GA}). Therefore, removing this constraint is not only mathematically desirable but also physically relevant. The main goal of this paper is to establish a global well-posedness theory without imposing any additional condition on the viscosity coefficients. In particular, we introduce a new scaling-invariant quantity compatible with \eqref{sca}, and our analysis completely removes the assumption $2\mu>\lambda$ required in \cite{LT23}. These features constitute the main novelties of the present work.

Before stating our main result precisely, we describe the notation throughout. The material derivative of the velocity and the vorticity are defined by
\begin{align*}
\dot{u} \triangleq u_{t}+u \cdot \nabla u, \quad  \operatorname{curl} u\triangleq\nabla \times u.
\end{align*}
For an integer $k \geq 0$ and $1 \leq q \leq \infty$, the standard Sobolev spaces are defined as
\begin{align*}
&D^{k,q}(\mathbb{R}^{3}) = \{ u \in L^1_{\loc}(\mathbb{R}^{3}) \mid \|\nabla^k u\|_{L^q} < \infty \}, \quad \|u\|_{D^{k,q}} = \|\nabla^k u\|_{L^q},\quad \|  u \|_{L^{q}}=\|u \|_{L^{q}(\mathbb{R}^{3})}, \\
&W^{k,q}(\mathbb{R}^{3}) = L^q(\mathbb{R}^{3}) \cap D^{k,q}(\mathbb{R}^{3}), \quad D^k(\mathbb{R}^{3}) = D^{k,2}(\mathbb{R}^{3}), \quad H^k(\mathbb{R}^{3}) = W^{k,2}(\mathbb{R}^{3}).
\end{align*}
For simplicity, we set
\begin{align}\label{z1}
\int \cdot \, d x \triangleq \int_{\mathbb{R}^{3}} \cdot \, d x, \quad \kappa = c_{v} = R = 1.
\end{align}

Next, we give the definition of the strong solution to \eqref{1}--\eqref{4} throughout this paper.
\begin{definition}
(Strong solution) For  $T>0$, $(\rho, u, \theta, d)$  is called a strong solution to \eqref{1}--\eqref{4} in  $\mathbb{R}^{3} \times[0, T]$, if for some  $q \in(3,6]$,
\begin{align*}
&\rho \in C ([0, T] ; H^{1} \cap W^{1, q} ), \quad(u, \theta) \in C ([0, T] ; D_{0}^{1} \cap D^{2} ) \cap L^{2} (0, T ; D^{2, q} ),\\
&\nabla d \in C ([0, T] ; H^{2} ) \cap L^{2} (0, T ; H^{3} ), \quad \rho_{t} \in C ([0, T] ; L^{2} \cap L^{q} ), \quad (u_{t}, \theta_{t} ) \in L^{2} (0, T ; D_{0}^{1} ), \\
& (\sqrt{\rho} u_{t}, \sqrt{\rho} \theta_{t} ) \in L^{\infty} (0, T ; L^{2} ), \quad d_{t} \in C ([0, T] ; H^{1} ) \cap L^{2} (0, T ; H^{2} ), \quad|d|=1,
\end{align*}
and $(\rho, u, \theta, d)$ satisfies system \eqref{1} a.e. in $\mathbb{R}^{3} \times(0, T]$.
\end{definition}

The main result  in this paper can be stated as follows.

\begin{theorem}\label{thm1}
Let $q\in(3,6]$ and assume that the initial data $(\rho_{0}\ge 0,u_{0},\theta_{0}\ge 0,d_{0})$ satisfy
\begin{align}\label{01}
\rho_{0} \in H^{1} \cap W^{1,q},~
(\sqrt{\rho_{0}} u_{0}, \sqrt{\rho_{0}} \theta_{0} ) \in L^{2},~
(u_{0}, \theta_{0}) \in D_{0}^{1} \cap D^{2},~
\nabla d_{0} \in H^{2},~ |d_{0}|=1,
\end{align}
and the compatibility condition
\begin{align}\label{com}
\begin{cases}
-\mu \Delta u_{0}-(\lambda+\mu) \nabla \operatorname{div} u_{0}
+\nabla  ( \rho_{0} \theta_{0} )
+\nabla d_{0} \cdot \Delta d_{0}
=\sqrt{\rho_{0}} g_{1}, \\[1mm]
-  \Delta \theta_{0}
-\mathcal{Q}(\nabla u_{0})
-|\Delta d_{0}+|\nabla d_{0}|^{2} d_{0} |^{2}
=\sqrt{\rho_{0}}  g_{2},
\end{cases}
\end{align}
for some $g_{1},g_{2} \in L^{2}(\mathbb{R}^{3})$. There exists a small positive constant $\varepsilon_{0}>0$ depending only on the parameters $R, \mu, \lambda, \kappa$, and $c_{v}$ but independent of the initial data, such that if
\begin{align}\label{sca1}
S_{0}\triangleq \big(1+\bar{\rho}+\tfrac{1}{\bar{\rho}}\big)\,S_{0}^{\prime}S_{0}^{\prime\prime} \leq \varepsilon_{0},
\end{align}
then the Cauchy problem \eqref{1}--\eqref{4} admits a unique global strong solution $(\rho,u,\theta,d)$. Here $\bar{\rho}\triangleq \|\rho_{0} \|_{L^{\infty}}$,
\begin{gather*}
	S_{0}^{\prime}\triangleq \|\rho_{0}\|_{L^{3}}+(\bar{\rho}^{2}+\bar{\rho}) (\|\sqrt{\rho_{0}}u_{0}\|_{L^{2}}^{2}+\|\nabla d_{0}\|_{L^{2}}^{2}),\\[2mm]
	S_{0}^{\prime\prime}\triangleq \|\nabla u_{0}\|_{L^{2}}^{2}+(\bar{\rho}+1) \| \sqrt{\rho_{0}}\theta_{0}\|_{L^{2}}^{2}+\| \nabla^{2} d_{0}\|_{L^{2}}^{2}+\| \nabla d_{0}\|_{L^{4}}^{4}.
\end{gather*}
\end{theorem}

\begin{remark}
It is not hard to check that the quantity $S_{0}$ in \eqref{sca1} is invariant under the scaling \eqref{sca}. Compared with \cite{LT23}, our result removes the additional viscosity constraint $2\mu>\lambda$ and also reveals that the scaling invariant quantity is not unique even under the same scaling transformation \eqref{sca}.
\end{remark}

\begin{remark}
The scaling-invariant quantity in \eqref{sca1} is rather complicated, and whether there exists a simpler scaling-invariant quantity remains an interesting question for further investigation.
\end{remark}

\begin{remark}
The nematic liquid crystal orientation field plays a significant role in the present analysis. Let $k>0$ and consider a general rescaling of space, time and unknowns of the form
\begin{align*}
\rho^{k}(x,t) \triangleq k^{a}\rho(k x,k^{\gamma} t),~
u^{k}(x,t) \triangleq k^{b}u(k x,k^{\gamma} t), ~
\theta^{k}(x,t) \triangleq k^{c}\theta(k x,k^{\gamma} t), ~
d^{k}(x,t) \triangleq k^{d}d(k x,k^{\gamma} t),
\end{align*}
where $a$, $b$, $c$, $d$, $\gamma\in\mathbb{R}$ are to be determined. A straightforward computation shows that the system \eqref{1} admits a unique nontrivial scaling-invariant transformation given by \eqref{sca}. In contrast, the non-isentropic Navier--Stokes equations \cite{L20,W25} admit a family of scaling-invariant transformations. Here, the presence of the orientation field rules out all other scaling structures.
\end{remark}

We mainly use a bootstrap argument and the blow-up criterion \eqref{up} to prove Theorem \ref{thm1}. It should be emphasized that the approaches in \cite{LZ21,LZ26,SZsub} do not apply in the present framework, since their smallness assumptions are imposed directly on the initial data and are not compatible with the scaling-invariant framework. In addition, the analysis in \cite{LT23} relies crucially on the extra viscosity condition $2\mu>\lambda$. This restriction is mainly used to derive a higher-integrability estimate for the velocity field, which is crucial in their argument. Hence, in order to derive the required \emph{a priori} estimates under the scaling-invariant assumption \eqref{sca1}, without any additional restriction on the viscosity coefficients, several new observations are needed.

First, the basic energy estimate does not provide the dissipation bound
\begin{align*}
\int_{0}^{T}\big(\|\nabla u\|_{L^{2}}^{2}
+\|\nabla^{2} d\|_{L^{2}}^{2}\big)\,dt,
\end{align*}
and we have to deal with $\|\rho\|_{L_t^\infty L_x^{3}}$ and
$\|\nabla \theta\|_{L_t^{2} L_x^{2}}$
(see Lemma \ref{lem1}).
Motivated by \cite{W25,L20}, we derive an $L_t^\infty L_x^{3}$ estimate for $\rho$ (see Lemma \ref{lem2}) and close the bound for $\|\rho\|_{L_t^\infty L_x^{3}}$ (see Corollary \ref{co3}).
In particular, we introduce the quantity
\begin{align*}
(\bar{\rho}^{2}+\bar{\rho})
\sup_{t\in[0,T]}\|\rho(t)\|_{L^{3}}
\int_{0}^{T}\|\nabla\theta\|_{L^{2}}^{2}\,dt,
\end{align*}
and assume that it is sufficiently small; see \eqref{p1}. The next step is to estimate $\|\nabla\theta\|_{L_t^{2} L_x^{2}}$.  Here, we work directly with the temperature equation, rather than the total energy equation used in \cite{LT23}. This strategy requires control of $\sqrt{\rho}\dot{u}$ in $L_t^{2} L_x^{2}$ (see Lemma \ref{lem3}).
By multiplying \eqref{1}$_2$ by $u_t$ and performing delicate energy estimates, we obtain the desired estimate for $\sqrt{\rho}\dot{u}$.
We emphasize that the coupling term $\nabla d\cdot \Delta d$ generates a trouble term, denoted by $K_{2}(t)$; see \eqref{xxxx0}. To handle the effects of the orientation field, we derive suitable estimates from \eqref{1}$_4$ (see Lemma \ref{lem5}). To close the bounds for
$\|\nabla u\|_{L_t^\infty L_x^{2}}$,
$\|\nabla^{2} d\|_{L_t^\infty L_x^{2}}$,
and $\|\nabla\theta\|_{L_t^{2} L_x^{2}}$,
we further introduce the quantity
\begin{align*}
\sup_{t\in[0,T]}\big(\|\nabla u(t)\|_{L^{2}}^{2}+(1+\bar{\rho}+\bar{\rho}^{2})\|\nabla^{2} d(t)\|_{L^{2}}^{2}\big)\int_{0}^{T}\big(\|\nabla u\|_{L^{2}}^{2}+\|\nabla^{2} d\|_{L^{2}}^{2}\big)\,dt,
\end{align*}
and assume that it is sufficiently small; see \eqref{p1}.
This allows us to close the estimates (see Corollary \ref{co4}).
Finally, we complete the bootstrap argument.
Combining the estimates in Lemmas \ref{lem1}--\ref{lem5}, we obtain a uniform upper bound for the density and the desired control of $S_t$, provided that the initial scaling-invariant quantity $N_{0}$ is sufficiently small (see Lemmas \ref{co1} and \ref{co2}).
In particular, in order to construct a complete scaling-invariant functional compatible with both the \emph{a priori} estimates and the bootstrap scheme, we multiply \eqref{c9} by $\big(1+\bar{\rho}+\tfrac{1}{\bar{\rho}}\big)$.

The rest of the paper is organized as follows.  Section \ref{sec2} is devoted to establishing \emph{a priori} estimates, which are necessary for the proof of Theorem \ref{thm1}. The proof of Theorem \ref{thm1} will be done in Section \ref{sec3}.

\section{\emph{A priori} estimates}\label{sec2}

In this section we establish the necessary \emph{a priori} estimates and collect several auxiliary lemmas needed to prove the global existence and uniqueness of strong solutions to \eqref{1}--\eqref{4}.
We begin with the local existence and uniqueness of strong solutions in \cite{FLN} and a Serrin-type blow-up criterion in \cite{SZ25}.

\begin{lemma}\label{local}
(Local well-posedness and blow-up criterion) Assume that the initial data
$(\rho_{0} \geq 0, u_{0}, \theta_{0} \geq 0, d_{0})$ satisfies \eqref{01} and \eqref{com} in Theorem \ref{thm1}. Then there exists a positive time $T_{0}>0$ and  problem \eqref{1}--\eqref{4} admits a unique strong solution $(  \rho, u, \theta, d)$ in  $\mathbb{R}^{3} \times\left(0, T_{0}\right)$. Furthermore, if $T_{\max}$  is the maximal time of existence for local strong solution, then either $T_{\max}=\infty$ or
	\begin{align}\label{up}
		\lim _{T \rightarrow T_{\max}}\left(\|\rho\|_{L^{\infty}\left(0, T ; L^{\infty}\right)}+\|u\|_{L^{s_{1}}\left(0, T ; L^{r_{1}}\right)}+\|\nabla d\|_{L^{s_{2}}\left(0, T ; L^{r_{2}}\right)}\right)=\infty,
	\end{align}
	with  $r_{i}$  and  $s_{i}$  satisfying
	\begin{align*}
		\frac{2}{s_{i}}+\frac{3}{r_{i}} \leq 1, \quad s_{i}>1, \quad 3<r_{i} \leq \infty, \quad i=1,2.
	\end{align*}
\end{lemma}
In what follows, let $(\rho,u,\theta,d)$ be a local strong solution to \eqref{1}--\eqref{4} in $\mathbb{R}^{3}\times(0,T]$ for some $T>0$.
For later use, we introduce
\begin{align*}
	S_{t} \triangleq \bigg(1+\bar{\rho}+\frac{1}{\bar{\rho}}\bigg)S_{t}^{\prime}S_{t}^{\prime\prime},
\end{align*}
where
\begin{gather*}
S_{t}^{\prime}\triangleq \|\rho(t)\|_{L^{3}}+( \bar{\rho}^{2}+\bar{\rho}) ( \| \sqrt{\rho}u(t)\|_{L^{2}}^{2}+\|\nabla d(t)\|_{L^{2}}^{2} ), \\[2mm]
S_{t}^{\prime\prime}\triangleq \|\nabla u(t)\|_{L^{2}}^{2}+(\bar{\rho}+1)\|\sqrt{\rho} \theta(t)\|_{L^{2}}^{2}+\| \nabla^{2} d(t)\|_{L^{2}}^{2}+\| \nabla d(t)\|_{L^{4}}^{4}.
\end{gather*}

The following bootstrap argument (see, e.g., \cite[Proposition 1.21]{T06}) plays a crucial role in the proof of Theorem \ref{thm1}.
\begin{proposition}\label{pro}
	Under the hypotheses of Theorem \ref{thm1}, let $\varepsilon_{0}>0$ be the small constant appearing in \eqref{sca1}.
	Then there exist positive constants $C_{1}$ and $\varepsilon_{2}$ such that:
	(i) $C_{1}$ and $\varepsilon_{2}$ depend only on the physical parameters, but are independent of the initial data, $\bar{\rho}$, and $T$;
	(ii) $\varepsilon_{2}$ is further determined by \eqref{smm1}, \eqref{sm1}, \eqref{sm2}, \eqref{sm4}, \eqref{sm5}, \eqref{sm6}, \eqref{sm11}, \eqref{sm7}, \eqref{sm9}, and \eqref{s4}.
Moreover, if for some $\varepsilon$ satisfying $C_{1}\varepsilon_{0}<\varepsilon \le \varepsilon_{2}$, one has, for all $t\in[0,T]$,
\begin{align}\label{p1}
\rho(t) \leq 2\bar{\rho},\quad \mathcal{B}(t)\leq 2\varepsilon,
\end{align}
then
\begin{align}\label{p2}
\rho(t) \leq \frac{3}{2}\bar{\rho},\quad \mathcal{B}(t)\leq \frac{3}{2}\varepsilon,
\end{align}
where
\begin{align*}
\mathcal{B}(t) &\triangleq S_{t} + (\bar{\rho}^{2}+\bar{\rho})
\int_{0}^{T} \|\nabla \theta\|_{L^{2}}^{2}\, dt
\sup_{t\in[0,T]} \|\rho(t)\|_{L^{3}}\\
&\quad+ \sup_{t\in[0,T]}\big[\|\nabla u(t)\|_{L^{2}}^{2}+\big(1+\bar{\rho}+\bar{\rho}^{2}\big)\|\nabla^{2} d(s) \|_{L^{2}}^{2} \big]  \int_{0}^{T}\big(\|\nabla u\|_{L^{2}}^{2} +\|\nabla^{2} d \|_{L^{2}}^{2}\big)\, d t.
\end{align*}
\end{proposition}

Throughout the rest of the paper, $C$ denotes a generic positive constant independent of the initial data, $T$, $\varepsilon_{0}$, $\varepsilon$, and $\bar{\rho}$, whose value may vary from line to line.

As a first step, we derive a basic energy estimate.
\begin{lemma}\label{lem1}
Under the hypotheses of Proposition \ref{pro},  it holds that
\begin{align}\label{x0}
& \sup _{t \in[0, T]}\big(\|\sqrt{\rho} u(t)\|_{L^{2}}^{2}+\|\nabla d(t)\|_{L^{2}}^{2}\big) +\int_{0}^{T}\big(\|\nabla u\|_{L^{2}}^{2}+ \|d_{t} \|_{L^{2}}^{2}+ \|\nabla^{2} d \|_{L^{2}}^{2}\big)\, d t \notag\\
&\leq  C\big(\|\sqrt{\rho_{0}} u_{0} \|_{L^{2}}^{2}+ \|\nabla d_{0} \|_{L^{2}}^{2}\big)+C\sup _{t \in[0, T]}\|\rho(t)\|_{L^{3}} \int_{0}^{T}\|\nabla \theta\|_{L^{2}}^{2}\, d t\sup _{t \in[0, T]}\|\rho(t)\|_{L^{3}}.
\end{align}
\end{lemma}
\begin{proof}
Multiplying \eqref{1}$_{2}$ by $u$ and integrating the resultant over  $\mathbb{R}^{3}$, we obtain that
\begin{align*}
&\frac{1}{2} \frac{d}{d t}\|\sqrt{\rho} u\|_{L^{2}}^{2}+\mu\|\nabla u\|_{L^{2}}^{2}+(\mu+\lambda)\|\operatorname{div} u\|_{L^{2}}^{2}  \\
  &=  \int  \rho \theta  \operatorname{div}u \,  d x-\int \Delta d \cdot \nabla d \cdot u \, d x \notag\\
	&\leq  C\|\rho\|_{L^{3}}\|\nabla \theta\|_{L^{2}}\|\operatorname{div} u\|_{L^{2}}+C\|\nabla u\|_{L^{2}}\|\nabla d\|_{L^{3}} \|\nabla^{2} d \|_{L^{2}} \notag\\[2mm]
	&\leq  (\mu+\lambda)\|\operatorname{div} u\|_{L^{2}}^{2}+C\|\rho\|_{L^{3}}^{2}\|\nabla \theta\|_{L^{2}}^{2}+\frac{\mu}{2}\|\nabla u\|_{L^{2}}^{2}+C\|\nabla d\|_{L^{3}}^{2} \|\nabla^{2} d \|_{L^{2}}^{2},
\end{align*}
which yields that
\begin{align}\label{x1}
\frac{d}{d t}\|\sqrt{\rho} u\|_{L^{2}}^{2}+ \|\nabla u\|_{L^{2}}^{2}\leq C\|\rho\|_{L^{3}}^{2}\|\nabla \theta\|_{L^{2}}^{2}+C\|\nabla d\|_{L^{3}}^{2} \|\nabla^{2} d \|_{L^{2}}^{2}.
\end{align}

Using \eqref{1}$_{4}$ and H\"{o}lder's inequality, it follows that
\begin{align}
	\frac{d}{d t}\|\nabla d\|_{2}^{2}+\|d_{t} \|_{L^{2}}^{2}+ \|\nabla^{2} d \|_{L^{2}}^{2}  & =\int |d_{t}-\Delta d |^{2} \, d x=\int |u \cdot \nabla d-|\nabla d|^{2} d |^{2} \, d x \notag\\
	& \leq C\big(\|u\|_{L^{6}}^{2}\|\nabla d\|_{L^{3}}^{2}+\|\nabla d\|_{L^{4}}^{4}\big) \notag\\[2mm]
	&\leq C\|\nabla d\|_{L^{3}}^{2}\big(\|\nabla u\|_{L^{2}}^{2}+ \|\nabla^{2} d \|_{L^{2}}^{2}\big),
\end{align}
which combined with \eqref{x1} and Sobolev's inequality implies that
\begin{align*}
 &\frac{d}{d t}\big(\|\sqrt{\rho} u\|_{L^{2}}^{2}+\|\nabla d\|_{L^{2}}^{2}\big)+ \|\nabla u\|_{L^{2}}^{2}+ \|d_{t} \|_{L^{2}}^{2}+ \|\nabla^{2} d \|_{L^{2}}^{2} \\
	&\leq   C\|\rho\|_{L^{3}}^{2}\|\nabla \theta\|_{L^{2}}^{2}+C\|\nabla d\|_{L^{3}}^{2}\big(\|\nabla u\|_{L^{2}}^{2}+ \|\nabla^{2} d \|_{L^{2}}^{2}\big) \\[2mm]
	&  \leq   C\|\rho\|_{L^{3}}^{2}\|\nabla \theta\|_{L^{2}}^{2}+C\|\nabla d\|_{L^{2}}\|\nabla^{2} d\|_{L^{2}} \big(\|\nabla u\|_{L^{2}}^{2}+ \|\nabla^{2} d \|_{L^{2}}^{2}\big).
\end{align*}
Integrating the above inequality over $(0,T)$, we have
\begin{align*}
&\sup _{t \in[0, T]}\big(\|\sqrt{\rho} u(t)\|_{L^{2}}^{2}+\|\nabla d(t)\|_{L^{2}}^{2}\big)+\int_{0}^{T}\big(\|\nabla u\|_{L^{2}}^{2}+ \|d_{t} \|_{L^{2}}^{2}+ \|\nabla^{2} d \|_{L^{2}}^{2}\big)\, d t\notag\\
&\leq  C\big( \|\sqrt{\rho_{0}} u_{0} \|_{L^{2}}^{2}+ \|\nabla d_{0} \|_{L^{2}}^{2}\big)+C\sup _{t \in[0, T]}(\|\nabla d(t)\|_{L^{2}}\|\nabla^{2} d(t)\|_{L^{2}}) \int_{0}^{T}\big(\|\nabla u\|_{L^{2}}^{2}+ \|\nabla^{2} d \|_{L^{2}}^{2}\big)\,  dt\notag\\
&\quad +C\sup _{t \in[0, T]}\|\rho(t)\|_{L^{3}} \int_{0}^{T}\|\nabla \theta\|_{L^{2}}^{2}\,  d t\sup _{t \in[0, T]}\|\rho(t)\|_{L^{3}},
\end{align*}
which leads to the desired \eqref{x0} provided that $\varepsilon>0$ in \eqref{p1} is sufficiently small so that
\begin{align}\label{smm1}
C\sup _{t \in[0, T]}\big(\|\nabla d(t)\|_{L^{2}}\|\nabla^{2} d(t)\|_{L^{2}}\big)\leq \frac{1}{2}.
\end{align}
This concludes the proof.
\end{proof}

To estimate $\|\rho\|_{L_{t}^{\infty}L_{x}^{3}}$, we prove the following lemma.
\begin{lemma}\label{lem2}
Under the hypotheses of Proposition \ref{pro},  it holds that
\begin{align}\label{xx1}
	 \sup _{t \in[0, T]} \|\rho(t)\|_{L^{3}} \leq C \|\rho_{0} \|_{L^{3}}+C(\bar{\rho}^{2}+\bar{\rho}) \int_{0}^{T}\big(\|\nabla u\|_{L^{2}}^{2}+\|\nabla^{2} d\|_{L^{2}}^{2}\big) \, d t.
\end{align}
\end{lemma}

\begin{proof}
Applying the operator  $(-\Delta)^{-1} \operatorname{div}$  to  \eqref{1}$_{2}$ yields that
\begin{align}\label{xx2}
 (-\Delta)^{-1} \operatorname{div}(\rho u)_{t} +(-\Delta)^{-1} \operatorname{div} \operatorname{div}(\rho u \otimes u)+(2 \mu+\lambda) \operatorname{div} u-P= -(-\Delta)^{-1} \operatorname{div} \operatorname{div}M(d),
\end{align}
due to
\begin{align*}
 \Delta d \cdot \nabla d= \operatorname{div}\left(\nabla d \odot \nabla d-\frac{1}{2}|\nabla d|^{2} \mathbb{I}_{3}\right),
\end{align*}
where
\begin{align*}
M(d) \triangleq \nabla d \odot \nabla d-\frac{1}{2}|\nabla d|^{2} \mathbb{I}_{3}, \quad  \nabla d \odot \nabla d \triangleq \sum_{k=1}^{3} \partial_{i} d^{k} \partial_{j} d^{k}, \quad \mathbb{I}_{3} \triangleq 3 \times 3 \text { unit matrix. }
\end{align*}

It follows from \eqref{1}$_{1}$ that
\begin{align}\label{xx3}
\partial_{t} \rho^{3}+\operatorname{div} (u \rho^{3} )+2 \operatorname{div} u \rho^{3}=0.
\end{align}
Multiplying \eqref{xx2} by  $\rho^{3}$  and using \eqref{xx3}, we find that
\begin{align}\label{xx4}
	&\rho^{3}(-\Delta)^{-1} \operatorname{div}(\rho u)_{t}+\rho^{3} (-\Delta)^{-1} \operatorname{div} \operatorname{div}(\rho u \otimes u)-\frac{2 \mu+\lambda}{2}\big(\partial_{t} \rho^{3}+\operatorname{div} (u \rho^{3} )\big)-\rho^{3}  P  \notag\\
	&=-\rho^{3}(-\Delta)^{-1} \operatorname{div} \operatorname{div}M(d).
\end{align}
Integration by parts together with \eqref{xx3} yields that
\begin{align*}
	&\int \rho^{3}  (-\Delta)^{-1} \operatorname{div}(\rho u)_{t}\, d x\notag\\
	&=  \frac{d}{d t} \int \rho^{3}  (-\Delta)^{-1} \operatorname{div}(\rho u)\, d x+\int\big(\operatorname{div} (\rho^{3} u )+2 \operatorname{div} u \rho^{3}\big)  (-\Delta)^{-1} \operatorname{div}(\rho u) \, d x \notag\\
	&=   \frac{d}{d t} \int \rho^{3}  (-\Delta)^{-1} \operatorname{div}(\rho u) \, d x+\int\Big[2 \operatorname{div} u \rho^{3}  (-\Delta)^{-1} \operatorname{div}(\rho u)-\rho^{3} u \cdot \nabla (-\Delta)^{-1} \operatorname{div}(\rho u)\Big]\, d x,
\end{align*}
which combined with \eqref{xx4}, H\"{o}lder's inequality, and Sobolev's inequality indicates that
\begin{align*}
	&\frac{d}{d t} \int\Big(\frac{2 \mu+\lambda}{2}-(-\Delta)^{-1} \operatorname{div}(\rho u)\Big) \rho^{3}\, d x+\int \rho^{3}  P  \, d x\notag\\
	&=\int\Big[  \rho^{3} (-\Delta)^{-1} \operatorname{div} \operatorname{div}(\rho u \otimes u)-\rho^{3}u \cdot \nabla  (-\Delta)^{-1} \operatorname{div}(\rho u)+2 \operatorname{div} u \rho^{3}  (-\Delta)^{-1}\operatorname{div}(\rho u)\Big]  \, d x\notag\\
	&\quad +\int \rho^{3} (-\Delta)^{-1} \operatorname{div} \operatorname{div} M(d)  \, d x\notag\\
	&\leq C\|\rho\|_{L^{\infty}}^{2}\|\rho\|_{L^{3}}\|\rho u^{2}\|_{L^{\frac{3}{2}}}+C\|\rho\|_{L^{\infty}}^{2}\|\rho\|_{L^{3}}\|\nabla u\|_{L^{2}}\|\rho u\|_{L^{2}} + C\|\rho\|_{L^{\infty}}\|\rho\|_{L^{3}}^{2} \|  (-\Delta)^{-1}\operatorname{div} \operatorname{div}M(d) \|_{L^{3}} \notag\\[2mm]
	&\leq C(\bar{\rho}^{2}+\bar{\rho}) \|\rho\|_{L^{3}}^{2}\big(\|\nabla u\|_{L^{2}}^{2}+\|\nabla^{2} d\|_{L^{2}}^{2}\big).
\end{align*}
Integrating the above inequality  over $(  0, T )$, one can deduce that
\begin{align}\label{xx7}
	&\sup _{t \in[0, T]} \|\rho(t)\|_{L^{3}}^{3}+\int_{0}^{T}\int\rho^{3} P \,  dx \, dt\notag\\
	&\leq   C\|\rho_{0}\|_{L^{3}}^{3}+C \bar{\rho}^{\frac{3}{4}} \sup _{t \in[0, T]}\Big(\|\sqrt{\rho} u(t)\|_{L^{2}}^{\frac{1}{2}}\|\nabla u(t)\|_{L^{2}}^{\frac{1}{2}}\Big) \|\rho(t) \|_{L^{3}}^{3}\notag\\
	& \quad +C(\bar{\rho}^{2}+\bar{\rho}) \sup _{t \in[0, T]}\|\rho(t)\|_{L^{3}}^{2}\int_{0}^{T}\big(\|\nabla u\|_{L^{2}}^{2}+\|\nabla^{2} d\|_{L^{2}}^{2}\big)\, d t,
\end{align}
where we have used
\begin{align}\label{xu0}
	\|(-\Delta)^{-1} \operatorname{div}(\rho u)\|_{L^{\infty}} & \leq C\|(-\Delta)^{-1} \operatorname{div}(\rho u)\|_{L^{6}}^{\frac{1}{2}}\|\nabla(-\Delta)^{-1} \operatorname{div}(\rho u)\|_{L^{6}}^{\frac{1}{2}} \notag\\
	& \leq C\|\rho u\|_{L^{2}}^{\frac{1}{2}}\|\rho u\|_{L^{6}}^{\frac{1}{2}} \notag\\
	& \leq C \bar{\rho}^{\frac{3}{4}}\|\sqrt{\rho} u\|_{L^{2}}^{\frac{1}{2}}\|\nabla u\|_{L^{2}}^{\frac{1}{2}}.
\end{align}

We choose $\varepsilon>0$ in \eqref{p1} small enough so that
\begin{align}\label{sm1}
C \bar{\rho}^{\frac{3}{4}}\sup _{t \in[0, T]}\Big(\|\sqrt{\rho} u(t)\|_{L^{2}}^{\frac{1}{2}}\|\nabla u(t)\|_{L^{2}}^{\frac{1}{2}}\Big)\leq \frac{1}{2}.
\end{align}
Then  \eqref{xx7} implies that
\begin{align*}
\sup _{t \in[0, T]} \|\rho(t)\|_{L^{3}}^{3} \leq C \|\rho_{0}\|_{L^{3}}^{3}+C(\bar{\rho}^{2}+\bar{\rho}) \sup _{t \in[0, T]}\|\rho(t)\|_{L^{3}}^{2} \int_{0}^{T}\big(\|\nabla u\|_{L^{2}}^{2}+\|\nabla^{2} d\|_{L^{2}}^{2}\big)\, d t,
\end{align*}
and the desired \eqref{xx1} follows by the Young inequality.
\end{proof}

As a consequence of Lemmas \ref{lem1} and \ref{lem2}, we obtain the following corollary.
\begin{corollary}\label{co3}
Under the hypotheses of Proposition \ref{pro},  it holds that
\begin{align}\label{c3}
&\sup _{t \in[0, T]} \Big[\|\rho(t)\|_{L^{3}}+( \bar{\rho}^{2}+\bar{\rho})\big( \| \sqrt{\rho}u(t)\|_{L^{2}}^{2}+\| \nabla d(t)\|_{L^{2}}^{2}\big)\Big]\notag\\
&\quad +( \bar{\rho}^{2}+\bar{\rho})\int_{0}^{T}\big(  \| \nabla u\|_{L^{2}}^{2}+\| \nabla^{2} d\|_{L^{2}}^{2}+\| d_{t}\|_{L^{2}}^{2}\big)\, d t  \leq C S_{0}^{\prime}.
\end{align}
\end{corollary}
\begin{proof}
	Multiplying \eqref{x0} by $2C( \bar{\rho}^{2}+\bar{\rho})$ and substituting the resultant into \eqref{xx1} implies that
	\begin{align*}
		& \sup _{t \in[0, T]} \Big[\|\rho(t)\|_{L^{3}}+C( \bar{\rho}^{2}+\bar{\rho}) \big( \| \sqrt{\rho}u(t)\|_{L^{2}}^{2}+\| \nabla d(t)\|_{L^{2}}^{2} \big)\Big]\notag\\
		&\quad   +C( \bar{\rho}^{2}+\bar{\rho})\int_{0}^{T} \big(  \| \nabla u\|_{L^{2}}^{2}+\| \nabla^{2} d\|_{L^{2}}^{2}+\| d_{t}\|_{L^{2}}^{2} \big) \, d t \notag\\
		&\leq C   \big(\|\rho_{0}\|_{L^{3}}+ ( \bar{\rho}^{2}+\bar{\rho})   \big(\| \sqrt{\rho_{0}}u_{0}\|_{L^{2}}^{2}+\| \nabla d_{0}\|_{L^{2}}^{2}\big) \big) \notag\\
		&\quad +C ( \bar{\rho}^{2}+\bar{\rho}) \sup _{ t \in[0, T]}\|\rho(t)\|_{L^{3}}\int_{0}^{T}\|\nabla \theta\|_{L^{2}}^{2} \, d t\sup _{ t \in[0, T]}\|\rho(t)\|_{L^{3}},
	\end{align*}
which yields \eqref{c3} provided $\varepsilon>0$ in \eqref{p1} is sufficiently small so that
	\begin{align}\label{sm6}
		C ( \bar{\rho}^{2}+\bar{\rho})\sup _{t \in[0, T]}\|\rho(t)\|_{L^{3}} \int_{0}^{T}\|\nabla \theta\|_{L^{2}}^{2} \, d t   \leq \frac{1}{2}.
	\end{align}
This concludes the proof.	
\end{proof}

Next, we estimate the term $\|\nabla \theta\|_{L_{t}^{2}L_{x}^{2}}$.
\begin{lemma}\label{lem3}
Under the hypotheses of Proposition \ref{pro},  it holds that
\begin{align}\label{xxx1}
	 &\sup _{t \in[0, T]}\|\sqrt\rho \theta(t)\|_{L^{2}}^{2}+\int_{0}^{T}\|\nabla \theta\|_{L^{2}}^{2} \, d t \notag\\
	 &\leq C \|\sqrt{\rho_{0}}\theta_{0}\|_{L^{2}}^{2}+C \sup _{t \in[0, T]}\big(\|\nabla u(t)\|_{L^{2}}^{2}\|\rho(t)\|_{L^{3}}\big) \int_{0}^{T}\|\sqrt{\rho} \dot{u}\|_{L^{2}}^{2}\, d t\notag\\
	&\quad  +C\int_{0}^{T}\|\nabla u\|_{L^{2}}^{2}\|\nabla^{2} d\|_{L^{2}}^{4}\, dt +C \int_{0}^{T}\|\nabla^{2} d \|_{L^{2}}^{3} \|\nabla^{3} d \|_{L^{2}}\, dt+C \int_{0}^{T}\|\nabla^{2} d \|_{L^{2}}^{7}\|\nabla d\|_{L^{2}}\, dt.
\end{align}
\end{lemma}
\begin{proof}
Multiplying \eqref{1}$_{3}$ by $\theta$ and integration by parts, one gets that
\begin{align}\label{xxx2}
\frac{1}{2} \frac{d}{d t} \|\sqrt\rho \theta\|_{L^{2}}^{2}+ \|\nabla \theta\|_{L^{2}}^{2}
	=-\int \rho \theta^{2} \operatorname{div} u \, dx+\int \mathcal{Q}(\nabla u) \theta \, dx+ \int  |\Delta d+|\nabla d|^{2} d |^{2} \theta \, dx \triangleq\sum_{i=1}^{3} I_{i}.
\end{align}
By  H\"{o}lder's inequality, Sobolev's inequality, and \eqref{p1}, we have
\begin{align}\label{xxx3}
I_{1} \leq C\|\rho\|_{L^{6}}\|\nabla u\|_{L^{2}}\|\theta^{2}\|_{L^{3}} \leq C\bar{\rho}^{\frac{1}{2}}\|\rho\|_{L^{3}}^{\frac{1}{2}}\|\nabla u\|_{L^{2}}\|\nabla \theta\|_{L^{2}}^{2}.
\end{align}
Applying $\operatorname{div}$ and $\operatorname{curl}$ to \eqref{1}$_{2}$, respectively, yields the elliptic system
\begin{align}\label{equ}
	\begin{cases}
	\Delta F=\operatorname{div}(\rho \dot{u}+\nabla d \cdot \Delta d),\\
	 \mu \Delta \operatorname{curl} u=\operatorname{curl} (\rho \dot{u}+\nabla d \cdot \Delta d),
		\end{cases}
\end{align}
where $	F \triangleq(2 \mu+\lambda) \operatorname{div} u-P$. Representing $F$ and $\operatorname{curl} u$ via elliptic operators and using the Sobolev embedding, we obtain that
	\begin{align} \label{eq2}
		\|F\|_{L^{3}}+\|\operatorname{curl} u\|_{L^{3}} \leq & C\| (-\Delta)^{-1} \operatorname{div}(\rho \dot{u}+\nabla d \cdot \Delta d)\|_{L^{3}}+C\| (-\Delta)^{-1} \operatorname{curl}  (\rho \dot{u}+\nabla d \cdot \Delta d)\|_{L^{3}} \notag\\
		\leq&  C\|\rho \dot{u}+\nabla d \cdot \Delta d\|_{L^{\frac{3}{2}}} \notag\\
		\leq &  C\| \rho \|_{L^{3}}^{\frac{1}{2}}\|\sqrt{\rho} \dot{u}\|_{L^{2}}+C\|\nabla^{2} d\|_{L^{2}}^{2}.
	\end{align}
By \eqref{eq2} and H\"older's  inequality, one has that
\begin{align}\label{xxx4}
	I_{2} \leq &  C\|\nabla \theta\|_{L^{2}}\|\nabla u\|_{L^{\frac{12}{5}}}^{2} \notag\\
	\leq &  C\|\nabla \theta\|_{L^{2}}\Big(\|\operatorname{div} u\|_{L^{\frac{12}{5}}}^{2}+ \|\operatorname{curl} u\|_{L^{\frac{12}{5}}}^{2}\Big)\notag\\
	\leq &  C\|\nabla \theta\|_{L^{2}}\big(\|\operatorname{div} u\|_{L^{2}}\|\operatorname{div} u\|_{L^{3}}+ \|\operatorname{curl} u\|_{L^{2}}\|\operatorname{curl} u\|_{L^{3}}\big)\notag\\
	\leq &  C\|\nabla \theta\|_{L^{2}}
	\|\nabla u\|_{L^{2}}\big( \|F\|_{L^{3}}+\|\rho \theta\|_{L^{3}} +\|\operatorname{curl} u\|_{L^{3}}\big) \notag\\	
	\leq & C\|\rho\|_{L^{3}}^{\frac{1}{2}}\|\nabla \theta\|_{L^{2}}\|\nabla u\|_{L^{2}}\|\sqrt{\rho} \dot{u}\|_{L^{2}}+C \bar{\rho}^{\frac{1}{2}}\|\rho\|_{L^{3}}^{\frac{1}{2}}\|\nabla u\|_{L^{2}}\|\nabla \theta\|_{L^{2}}^{2} + C\|\nabla \theta\|_{L^{2}}\|\nabla u\|_{L^{2}}\|\nabla^{2} d\|_{L^{2}}^{2}.
\end{align}
Noting that $|d|=1$ implies $|\Delta d+|\nabla d|^{2} d |^{2}=|\Delta d|^{2}-|\nabla d|^{4}$,  we estimate  $I_{3}$ as follows
\begin{align}\label{xxx5}
I_{3} &\leq   C\|\theta\|_{L^{6}}\|\Delta d\|_{L^{2}}\|\Delta d\|_{L^{3}} +C\|\nabla d\|_{L^{6}}\|\nabla d\|_{L^{6}}\|\nabla d\|_{L^{6}}\|\nabla d\|_{L^{3}}\|\theta\|_{L^{6}} \notag\\
	&\leq   C\|\nabla \theta\|_{L^{2}} \|\nabla^{2} d \|_{L^{2}}^{\frac{3}{2}} \|\nabla^{3} d \|_{L^{2}}^{\frac{1}{2}}+C \|\nabla^{2} d \|_{L^{2}}^{\frac{7}{2}}\|\nabla d\|_{L^{2}}^{\frac{1}{2}}\|\nabla \theta\|_{L^{2}}.
\end{align}

Substituting  \eqref{xxx3}, \eqref{xxx4}, and \eqref{xxx5} into \eqref{xxx2} and applying Young's inequality, we deduce that
\begin{align*}
	 &\frac{d}{d t} \|\sqrt\rho \theta\|_{L^{2}}^{2}+ \|\nabla \theta\|_{L^{2}}^{2} \\
	 &\leq C\|\rho\|_{L^{3}}\|\nabla u\|_{L^{2}}^{2}\|\sqrt{\rho} \dot{u}\|_{L^{2}}^{2}+C \bar{\rho}^{\frac{1}{2}}\|\rho\|_{L^{3}}^{\frac{1}{2}}\|\nabla u\|_{L^{2}}\|\nabla \theta\|_{L^{2}}^{2}+C \|\nabla u\|_{L^{2}}^{2}\|\nabla^{2} d\|_{L^{2}}^{4}\\[2mm]
	 &\quad +C \|\nabla^{2} d \|_{L^{2}}^{3} \|\nabla^{3} d \|_{L^{2}} +C \|\nabla^{2} d \|_{L^{2}}^{7}\|\nabla d\|_{L^{2}}.
\end{align*}
Integrating the above inequality over $(0,T)$ and choosing  $\varepsilon>0$ in \eqref{p1} sufficiently small so that
\begin{align}\label{sm2}
C \bar{\rho}^{\frac{1}{2}}\sup _{t \in[0, T]}\|\rho(t)\|_{L^{3}}^{\frac{1}{2}}\|\nabla u(t)\|_{L^{2}} \leq \frac{1}{2},
\end{align}
we obtain the desired \eqref{xxx1}.
\end{proof}

\begin{lemma}\label{lem4}
Under the hypotheses of Proposition \ref{pro}, it holds that
\begin{align}\label{xxxx0}
	 \sup _{t \in[0, T]}\|\nabla u(t)\|_{L^{2}}^{2}+\int_{0}^{T} \|\sqrt{\rho} \dot{u}\|_{L^{2}}^{2}\, d t &\leq C \big(\|\nabla u_{0}\|_{L^{2}}^{2}+\bar{\rho} \|\sqrt{\rho_{0}}\theta_{0}\|_{2}^{2}+\|\nabla d_{0}\|_{L^{4}}^{4}\big)+C \bar{\rho} \|\sqrt{\rho}\theta(t)\|_{2}^{2}+C\|\nabla d(t)\|_{L^{4}}^{4}  \notag\\
	&\quad +C \sup _{t \in[0, T]} K_{1}(t)  \int_{0}^{T}\|\nabla \theta\|_{L^{2}}^{2} \, d t+C\int_{0}^{T} K_{2}(t) \, d t,
\end{align}
where \begin{align*}
K_{1}(t)\triangleq\bar{\rho}+\bar{\rho}^{2}\|\rho(t)\|_{L^{3}}\|\nabla u(t)\|_{L^{2}}^{2},
\end{align*}
and
\begin{align*}
	K_{2}(t)&\triangleq \Big(\|\rho\|_{L^{3}}\|\nabla u\|_{L^{2}}^{2}\||\nabla d||\nabla^{2} d|\|_{L^{2}}^{2} + \bar{\rho}\|\nabla \theta\|_{L^{2}}\|\nabla d\|_{L^{2}}^{\frac{1}{2}} \|\nabla^{2} d \|_{L^{2}}^{\frac{1}{2}} \|\nabla d_{t} \|_{L^{2}} +\bar{\rho}\|\nabla  d\|_{L^{2}} \|\nabla^{2} d\|_{L^{2}}\|\nabla  d_{t}\|_{L^{2}}^{2} \\
	&\quad +\bar{\rho}\|\nabla u\|_{L^{2}}^{2}\|\nabla^{2} d\|_{L^{2}}^{4} +\bar{\rho} \|\nabla^{2} d\|_{L^{2}}^{3} \|\nabla^{3} d\|_{L^{2}} +\bar{\rho} \|\nabla^{2} d\|_{L^{2}}^{7} \|\nabla  d\|_{L^{2}}\\
	&\quad+\|\nabla d\|_{L^{2}}^{\frac{1}{2}} \|\nabla^{2} d \|_{L^{2}}^{\frac{1}{2}} \|\nabla d_{t} \|_{L^{2}}\||\nabla d||\nabla^{2} d|\|_{L^{2}}+ \bar{\rho}^{\frac{1}{2}}\| \rho\|_{L^{3}}^{\frac{1}{2}}\|\nabla \theta\|_{L^{2}}\|\nabla u\|_{L^{2}}\||\nabla d||\nabla^{2}d|\|_{L^{2}}   \\
	&\quad +\|\nabla u\|_{L^{2}}\|\nabla^{2}d\|_{L^{2}}^{2}   \||\nabla d||\nabla^{2}  d|\|_{L^{2}}+\|\nabla^{2} d\|_{L^{2}}^{\frac{3}{2}}\|\nabla^{3} d\|_{L^{2}}^{\frac{1}{2}}\||\nabla d||\nabla^{2} d|\|_{L^{2}}  \\
	&\quad +  \|\nabla^{2} d\|_{L^{2}}^{\frac{7}{2}}\|\nabla d\|_{L^{2}}^{\frac{1}{2}}\||\nabla d||\nabla^{2} d|\|_{L^{2}}+\|\nabla \theta\|_{L^{2}}\||\nabla d||\nabla^{2} d|\|_{L^{2}}\Big).
\end{align*}
\end{lemma}

\begin{proof}
Multiplying \eqref{1}$_{2}$ by $u_{t}$, integrating the resultant over $\mathbb{R}^3$, and using the definition of $F$, we derive
\begin{align}\label{xxxx1}
	&\frac{1}{2} \frac{d}{d t} \big(\mu\|\nabla u\|_{L^{2}}^{2}+(\mu+\lambda)\|\operatorname{div} u\|_{L^{2}}^{2}\big)+\|\sqrt{\rho} \dot{u}\|_{L^{2}}^{2} \notag\\
	&= \frac{d}{dt} \int  P \operatorname{div} u \, dx -\int P_{t}  \operatorname{div} u \, dx+\int \rho u \cdot \nabla  u \cdot \dot{u}\, dx+\int M(d)\cdot \nabla u_{t}\, d x\notag\\
	&=\frac{d}{d t} \int P \operatorname{div} u \, d x-\frac{1}{2(2 \mu+\lambda)} \frac{d}{d t}\|P\|_{L^{2}}^{2} -\frac{1}{2 \mu+\lambda} \int P_{t} F \, d x+\int \rho u \cdot \nabla  u \cdot \dot{u}\, dx+\int M(d)\cdot \nabla u_{t}\, d x\notag\\
	&=\frac{1}{2(2 \mu+\lambda)} \frac{d}{d t}\|P\|_{L^{2}}^{2}+\frac{1}{2 \mu+\lambda} \frac{d}{d t} \int P F \,  d x  -\frac{1}{2 \mu+\lambda} \int P_{t} F \, d x+\int \rho u \cdot \nabla  u \cdot \dot{u}\, dx+\int M(d)\cdot \nabla u_{t}\, d x\notag\\
	&\leq \frac{1}{2(2 \mu+\lambda)} \frac{d}{d t}\|P\|_{L^{2}}^{2}+\frac{1}{2 \mu+\lambda} \frac{d}{d t} \int P F \,  d x  -\frac{1}{2 \mu+\lambda} \int P_{t} F \, d x+\frac{1}{2}\|\sqrt{\rho} \dot{u}\|_{L^{2}}^{2}\notag\\
	&\quad +C\int  \rho|u|^{2}|\nabla u|^{2}\,  dx+\int M(d)\cdot \nabla u_{t}\, d x.
\end{align}
From  \eqref{1}$_{2}$, \eqref{1}$_{3}$, and \eqref{z1}, the pressure $P$ satisfies
\begin{align*}
	P_{t}=-\operatorname{div}(Pu)- P \operatorname{div} u+\Delta \theta+\mathcal{Q}(\nabla u)+|\Delta d+|\nabla d|^{2} d |^{2},
\end{align*}
which leads to
\begin{align}\label{xxxx2}
	\int P_{t} F  \, d x
		&=\int\big[\big(\mathcal{Q}(\nabla u)-P \operatorname{div} u+|\Delta d+|\nabla d|^{2} d |^{2}\big) F+( P u- \nabla \theta) \cdot \nabla F\big] \, d x.
\end{align}	
It follows from H\"{o}lder's inequality, \eqref{xxxx1}, and \eqref{xxxx2} that
\begin{align}\label{xxxx3}
	& \frac{1}{2}\frac{d}{d t}\bigg(\mu\|\operatorname{curl} u\|_{L^{2}}^{2}+\frac{\|F\|_{L^{2}}^{2}}{2 \mu+\lambda} \bigg)+ \frac{1}{2}\|\sqrt{\rho} \dot{u}\|_{L^{2}}^{2} \notag\\
	&\leq C \int  \rho|u|^{2}|\nabla u|^{2}\,  dx+  \int M(d): \nabla u_{t}\,  d x-\frac{1}{2 \mu+\lambda} \int ( P u- \nabla \theta) \cdot \nabla F \, d x \notag\\
	&\quad-\frac{1}{2 \mu+\lambda}\int \Big(\mathcal{Q}(\nabla u)-P \operatorname{div} u+|\Delta d+|\nabla d|^{2} d |^{2}\Big) F \, d x\triangleq \sum_{i=1}^{4} M_{i}.
\end{align}

The standard $L^{2}$-theory for the elliptic system \eqref{equ} directly gives that
\begin{align}\label{eq1}
	\|\nabla F\|_{L^{2}}+\|\nabla \operatorname{curl} u\|_{L^{2}} \leq C\Big(\bar{\rho}^{\frac{1}{2}}\|\sqrt{\rho} \dot{u}\|_{L^{2}}+\||\nabla d||\nabla^{2} d|\|_{L^{2}}\Big).
\end{align}
Moreover, one has that
\begin{align}\label{eq3}
	\|\nabla u\|_{L^{6}} & \leq C\big(\|\operatorname{curl} u\|_{L^{6}}+\|\operatorname{div} u\|_{L^{6}}\big) \notag\\
	& \leq C\big(\|\operatorname{curl} u\|_{L^{6}}+\|F\|_{L^{6}}+\|\rho \theta\|_{L^{6}} \big) \notag\\
	& \leq C\big(\|\nabla \operatorname{curl} u\|_{L^{2}}+\|\nabla F\|_{L^{2}}+\bar{\rho}\|\nabla \theta\|_{L^{2}} \big)\notag\\
	&\leq C\Big(\bar{\rho}^{\frac{1}{2}}\|\sqrt{\rho} \dot{u}\|_{L^{2}}+\bar{\rho}\|\nabla \theta\|_{L^{2}}+\||\nabla d||\nabla^{2} d|\|_{L^{2}}\Big).
\end{align}
By H\"{o}lder's  inequality,  Sobolev's inequality, \eqref{eq3}, and \eqref{eq1}, we estimate
\begin{align*}
	M_{1}  &\leq  C\|\rho\|_{L^{3}}\|u\|_{L^{6}}^{2}\|\nabla u\|_{L^{6}}^{2} \notag\\[2mm]
     &\leq  C \bar{\rho}\|\rho\|_{L^{3}}\|\nabla u\|_{L^{2}}^{2}\|\sqrt{\rho} \dot{u}\|_{L^{2}}^{2}+C \bar{\rho}^{2}\|\rho\|_{L^{3}}\|\nabla u\|_{L^{2}}^{2} \|\nabla \theta\|_{L^{2}}^{2}+C\|\rho\|_{L^{3}}\|\nabla u\|_{L^{2}}^{2}\||\nabla d||\nabla^{2} d|\|_{L^{2}}^{2},\\[2mm]
     M_{2} &= \frac{d}{d t} \int  M(d): \nabla u \, d x-\int  \mathcal{M}(d)_{t}: \nabla u \, d x \notag\\
     	& \leq \frac{d}{d t} \int  M(d): \nabla u \, d x+C\|\nabla d\|_{L^{3}} \|\nabla d_{t} \|_{L^{2}}\|\nabla u\|_{L^{6}} \notag\\
     	&  \leq \frac{d}{d t} \int  M(d): \nabla u \, d x+C\|\nabla d\|_{L^{2}}^{\frac{1}{2}} \|\nabla^{2} d \|_{L^{2}}^{\frac{1}{2}} \|\nabla d_{t} \|_{L^{2}}\Big(\bar{\rho}^{\frac{1}{2}}\|\sqrt{\rho} \dot{u}\|_{L^{2}}+\bar{\rho}\|\nabla \theta\|_{L^{2}}+\||\nabla d||\nabla^{2} d|\|_{L^{2}}\Big)\notag \\
     	&  \leq \frac{d}{d t} \int  M(d): \nabla u \, d x+C\bar{\rho}^{\frac{1}{2}}\|\sqrt{\rho} \dot{u}\|_{L^{2}}\|\nabla d\|_{L^{2}}^{\frac{1}{2}} \|\nabla^{2} d \|_{L^{2}}^{\frac{1}{2}} \|\nabla d_{t} \|_{L^{2}}+C\bar{\rho}\|\nabla \theta\|_{L^{2}}\|\nabla d\|_{L^{2}}^{\frac{1}{2}} \|\nabla^{2} d \|_{L^{2}}^{\frac{1}{2}} \|\nabla d_{t} \|_{L^{2}}\notag\\
     	&\quad +C\|\nabla d\|_{L^{2}}^{\frac{1}{2}} \|\nabla^{2} d \|_{L^{2}}^{\frac{1}{2}} \|\nabla d_{t} \|_{L^{2}}\||\nabla d||\nabla^{2} d|\|_{L^{2}},\\[2mm]
    M_{3} &\leq C\big( \|\rho\|_{L^{6}}\| u\|_{L^{6}}\|\theta\|_{L^{6}}\|\nabla F\|_{L^{2}}+  \|\nabla \theta\|_{L^{2}}\|\nabla F\|_{L^{2}}\big)\notag\\[2mm]
    & \leq C \bar{\rho}\| \rho\|_{L^{3}}^{\frac{1}{2}}\|\nabla \theta\|_{L^{2}}\|\nabla u\|_{L^{2}}\|\sqrt{\rho} \dot{u}\|_{L^{2}}+C\bar{\rho}^{\frac{1}{2}}\| \rho\|_{L^{3}}^{\frac{1}{2}}\|\nabla \theta\|_{L^{2}}\|\nabla u\|_{L^{2}}\||\nabla d||\nabla^{2}d|\|_{L^{2}}  \notag\\[2mm]	
    &\quad +C \bar{\rho}^{\frac{1}{2}}\|\sqrt{\rho} \dot{u}\|_{L^{2}}\|\nabla \theta\|_{L^{2}} +C\|\nabla \theta\|_{L^{2}}\||\nabla d||\nabla^{2} d|\|_{L^{2}}, \\[2mm]   	
	M_{4} &\leq C\big(\|\nabla F\|_{L^{2}}\|\nabla u\|_{L^{\frac{12}{5}}}^{2}+ \|\nabla u\|_{L^{2}}\|\rho \theta\|_{L^{3}}\|\nabla F\|_{L^{2}}+ \| \Delta d\|_{L^{3}}\| \Delta d\|_{L^{2}}\|\nabla F\|_{L^{2}}+ \| \nabla d\|_{L^{6}}^{3}\| \nabla d\|_{L^{3}}\|\nabla F\|_{L^{2}}\big)\notag\\
	&\leq C\|\nabla F\|_{L^{2}}\|\nabla u\|_{L^{2}}\Big( \| \rho \|_{L^{3}}^{\frac{1}{2}}\|\sqrt{\rho} \dot{u}\|_{L^{2}}+ \|\nabla^{2} d\|_{L^{2}}^{2}+ \bar{\rho}^{\frac{1}{2}} \| \rho \|_{L^{3}}^{\frac{1}{2}}\|\nabla\theta\|_{L^{2}}  \Big) \notag\\	
	&\quad+C \bar{\rho}^{\frac{1}{2}}\|\rho\|_{L^{3}}^{\frac{1}{2}}\|\nabla u\|_{L^{2}}\|\nabla F\|_{L^{2}}\|\nabla \theta\|_{L^{2}} +C\| \nabla^{2} d\|_{L^{2}}^{\frac{3}{2}}\| \nabla^{3} d\|_{L^{2}}^{\frac{1}{2}}\|\nabla F\|_{L^{2}}+ C\| \nabla^{2} d\|_{L^{2}}^{\frac{7}{2}}\| \nabla d\|_{L^{2}}^{\frac{1}{2}}\|\nabla F\|_{L^{2}}\notag\\
	& \leq C  \bar{\rho}^{\frac{1}{2}}\| \rho\|_{L^{3}}^{\frac{1}{2}}\| \nabla u\|_{L^{2}}\|\sqrt{\rho} \dot{u}\|_{L^{2}}^{2}+C \| \rho\|_{L^{3}}^{\frac{1}{2}}\| \nabla u\|_{L^{2}}\|\sqrt{\rho} \dot{u}\|_{L^{2}}\||\nabla d||\nabla^{2} d|\|_{L^{2}}\notag\\	
	&\quad +C \bar{\rho}\| \rho\|_{L^{3}}^{\frac{1}{2}}\|\nabla \theta\|_{L^{2}}\|\nabla u\|_{L^{2}}\|\sqrt{\rho} \dot{u}\|_{L^{2}}+C\bar{\rho}^{\frac{1}{2}}\| \rho\|_{L^{3}}^{\frac{1}{2}}\|\nabla \theta\|_{L^{2}}\|\nabla u\|_{L^{2}}\||\nabla d||\nabla^{2}d|\|_{L^{2}}  \notag\\	
	&\quad +C\bar{\rho}^{\frac{1}{2}}\|\nabla u\|_{L^{2}}\|\nabla^{2} d\|_{L^{2}}^{2}\|\sqrt{\rho} \dot{u}\|_{L^{2}}+C\|\nabla u\|_{L^{2}}\|\nabla^{2}d\|_{L^{2}}^{2}   \||\nabla d||\nabla^{2} d|\|_{L^{2}}   \notag\\	
	&\quad +C\bar{\rho}^{\frac{1}{2}}\|\sqrt{\rho} \dot{u}\|_{L^{2}}\|\nabla^{2} d\|_{L^{2}}^{\frac{3}{2}}\|\nabla^{3} d\|_{L^{2}}^{\frac{1}{2}}+C\|\nabla^{2} d\|_{L^{2}}^{\frac{3}{2}}\|\nabla^{3} d\|_{L^{2}}^{\frac{1}{2}}\||\nabla d||\nabla^{2} d|\|_{L^{2}}  \notag\\	
	&\quad +C \bar{\rho}^{\frac{1}{2}}\|\sqrt{\rho} \dot{u}\|_{L^{2}}\|\nabla^{2} d\|_{L^{2}}^{\frac{7}{2}}\|\nabla d\|_{L^{2}}^{\frac{1}{2}}+C\|\nabla^{2} d\|_{L^{2}}^{\frac{7}{2}}\|\nabla d\|_{L^{2}}^{\frac{1}{2}}\||\nabla d||\nabla^{2} d|\|_{L^{2}}.
\end{align*}
Substituting the above estimates on $M_i\ (i=1,2,3,4)$ into \eqref{xxxx3}  and applying Young's inequality, we deduce that
\begin{align}\label{xxxx9}
&\frac{1}{2} \frac{d}{d t}\bigg(\mu\|\operatorname{curl} u\|_{L^{2}}^{2}+\frac{\|F\|_{L^{2}}^{2}}{2 \mu+\lambda}-2\int  M(d): \nabla u \, d x \bigg)+\frac{1}{2}\|\sqrt{\rho} \dot{u}\|_{L^{2}}^{2}\notag\\
&\leq C  \bar{\rho}\|\rho\|_{L^{3}}\|\nabla u\|_{L^{2}}^{2}\|\sqrt{\rho} \dot{u}\|_{L^{2}}^{2}+C \bar{\rho}^{2}\|\rho\|_{L^{3}}\|\nabla u\|_{L^{2}}^{2} \|\nabla \theta\|_{L^{2}}^{2}+C\|\rho\|_{L^{3}}\|\nabla u\|_{L^{2}}^{2}\||\nabla d||\nabla^{2} d|\|_{L^{2}}^{2} \notag\\
&\quad +C \bar{\rho}^{\frac{1}{2}}\|\sqrt{\rho} \dot{u}\|_{L^{2}}\|\nabla d\|_{L^{2}}^{\frac{1}{2}} \|\nabla^{2} d \|_{L^{2}}^{\frac{1}{2}} \|\nabla d_{t} \|_{L^{2}}+C\bar{\rho}\|\nabla \theta\|_{L^{2}}\|\nabla d\|_{L^{2}}^{\frac{1}{2}} \|\nabla^{2} d \|_{L^{2}}^{\frac{1}{2}} \|\nabla d_{t} \|_{L^{2}} \notag\\
 &\quad +C\|\nabla d\|_{L^{2}}^{\frac{1}{2}} \|\nabla^{2} d \|_{L^{2}}^{\frac{1}{2}} \|\nabla d_{t} \|_{L^{2}}\||\nabla d||\nabla^{2} d|\|_{L^{2}}\notag\\
&\quad +C \bar{\rho}^{\frac{1}{2}}\| \rho\|_{L^{3}}^{\frac{1}{2}}\| \nabla u\|_{L^{2}}\|\sqrt{\rho} \dot{u}\|_{L^{2}}^{2}+ C\| \rho\|_{L^{3}}^{\frac{1}{2}}\| \nabla u\|_{L^{2}}\|\sqrt{\rho} \dot{u}\|_{L^{2}}\||\nabla d||\nabla^{2} d|\|_{L^{2}}\notag\\	
&\quad +C \bar{\rho}\| \rho\|_{L^{3}}^{\frac{1}{2}}\|\nabla \theta\|_{L^{2}}\|\nabla u\|_{L^{2}}\|\sqrt{\rho} \dot{u}\|_{L^{2}}+C\bar{\rho}^{\frac{1}{2}}\| \rho\|_{L^{3}}^{\frac{1}{2}}\|\nabla \theta\|_{L^{2}}\|\nabla u\|_{L^{2}}\||\nabla d||\nabla^{2}d|\|_{L^{2}}  \notag\\	
&\quad +C\bar{\rho}^{\frac{1}{2}}\|\nabla u\|_{L^{2}}\|\nabla^{2} d\|_{L^{2}}^{2}\|\sqrt{\rho} \dot{u}\|_{L^{2}}+C\|\nabla u\|_{L^{2}}\|\nabla^{2}d\|_{L^{2}}^{2}   \||\nabla d||\nabla^{2} d|\|_{L^{2}}   \notag\\	
&\quad +C\bar{\rho}^{\frac{1}{2}}\|\sqrt{\rho} \dot{u}\|_{L^{2}}\|\nabla^{2} d\|_{L^{2}}^{\frac{3}{2}}\|\nabla^{3} d\|_{L^{2}}^{\frac{1}{2}}+C\|\nabla^{2} d\|_{L^{2}}^{\frac{3}{2}}\|\nabla^{3} d\|_{L^{2}}^{\frac{1}{2}}\||\nabla d||\nabla^{2} d|\|_{L^{2}}  \notag\\	
&\quad +C \bar{\rho}^{\frac{1}{2}}\|\sqrt{\rho} \dot{u}\|_{L^{2}}\|\nabla^{2} d\|_{L^{2}}^{\frac{7}{2}}\|\nabla d\|_{L^{2}}^{\frac{1}{2}}+C\|\nabla^{2} d\|_{L^{2}}^{\frac{7}{2}}\|\nabla d\|_{L^{2}}^{\frac{1}{2}}\||\nabla d||\nabla^{2} d|\|_{L^{2}}  \notag\\	
&\quad +C \bar{\rho}^{\frac{1}{2}}\|\sqrt{\rho} \dot{u}\|_{L^{2}}\|\nabla \theta\|_{L^{2}} +C\|\nabla \theta\|_{L^{2}}\||\nabla d||\nabla^{2} d|\|_{L^{2}}\notag\\
&\leq \Big( C \bar{\rho}\|\rho\|_{L^{3}}\|\nabla u\|_{L^{2}}^{2}+ C\bar{\rho}^{\frac{1}{2}}\|\rho\|_{L^{3}}^{\frac{1}{2}} \|\nabla u\|_{L^{2}}+\frac{1}{8}\Big)\|\sqrt{\rho} \dot{u}\|_{L^{2}}^{2}+CK_{1}(t)\|\nabla \theta\|_{L^{2}}^{2}+CK_{2}(t).
\end{align}
For any $\delta>0$, Young's inequality implies that
\begin{align*}
	\left|\int M(d): \nabla u \, d x\right|
	\leq \delta \|\nabla u\|_{L^{2}}^{2} +C\|\nabla d\|_{L^{4}}^{4},
\end{align*}
and
\begin{align*}
	\|\nabla u\|_{L^{2}}^{2}
	\leq C \big(\|\operatorname{curl} u\|_{L^{2}}^{2}
	+\|F\|_{L^{2}}^{2}
	+\bar{\rho} \|\sqrt{\rho}\theta\|_{L^{2}}^{2}\big).
\end{align*}
We choose $\varepsilon>0$ in \eqref{p1} sufficiently small so that
\begin{align}\label{sm4}
	C\max\Big\{ \bar{\rho}\sup _{t \in[0, T]}(\|\rho(t)\|_{L^{3}}\|\nabla u(t)\|_{L^{2}}^{2}),~~
	\bar{\rho}^{\frac{1}{2}}\sup _{t \in[0, T]}\big(\|\rho(t)\|_{L^{3}}^{\frac{1}{2}} \|\nabla u(t)\|_{L^{2}}\big)\Big\} \leq \frac{1}{16}.
\end{align}
Then, integrating \eqref{xxxx9} over $(0,T)$ and choosing $\delta$ sufficiently small, we obtain \eqref{xxxx0}.
\end{proof}

To deal with the difficulties arising from the orientation field, we need the following lemma.
\begin{lemma}\label{lem5}
Under the hypotheses of Proposition \ref{pro},  it holds that
\begin{align}\label{xxxxx0}
		&\sup _{t \in[0, T]}\big( \|\nabla^{2} d(t) \|_{L^{2}}^{2}+\|\nabla d(t)\|_{L^{4}}^{4}\big) +\int_{0}^{T}\big( \|\nabla d_{t} \|_{L^{2}}^{2}+ \|\nabla^{3} d \|_{L^{2}}^{2}+\||\nabla d|| \nabla^{2} d| \|_{L^{2}}^{2}\big)\, d t \notag\\
		&\leq C\big( \|\nabla^{2} d_{0} \|_{L^{2}}^{2}+\|\nabla d_{0}\|_{L^{4}}^{4}\big)+ C \sup _{t \in[0, T]}(\|\nabla d(t)\|_{L^{2}}^{2}\|\nabla u(t)\|_{L^{2}}^{2}) \int_{0}^{T}\|\nabla u\|_{L^{2}}^{6}\, d t+C\int_{0}^{T} \|\nabla^{2} d \|_{L^{2}}^{6}\, d t.
\end{align}
\end{lemma}

\begin{proof}
Applying the operator $\nabla$ to \eqref{1}$_{4}$, we have
\begin{align}\label{xxxxx1}
\nabla d_{t}-\nabla \Delta d=-\nabla(u \cdot \nabla d)+ \nabla (  |\nabla d |^{2} d ).
\end{align}
A straightforward calculation shows that
\begin{align}\label{xxxxx2}
	& \frac{d}{d t} \|\nabla^{2} d \|_{L^{2}}^{2}+\|\nabla d_{t} \|_{L^{2}}^{2}+ \|\nabla^{3} d \|_{L^{2}}^{2} \notag\\
	&=  \int |\nabla d_{t}-\nabla \Delta d |^{2} \, d x \notag\\
	&=  \int |\nabla(u \cdot \nabla d)- \nabla ( |\nabla d |^{2} d ) |^{2} \, d x  \notag\\
	&\leq C  \||\nabla u| |\nabla d|\|_{L^{2}}^{2}+C\||u|  |\nabla^{2} d |\|_{L^{2}}^{2}+C\||\nabla d|  |\nabla^{2} d |\|_{L^{2}}^{2}+C\|\nabla d \|_{L^{6}}^{6} \notag\\[2mm]
	&\leq C \|\nabla u\|_{L^{2}}^{2}\|\nabla d\|_{L^{\infty}}^{2}+C \|u\|_{L^{6}}^{2} \|\nabla^{2} d \|_{L^{6}} \|\nabla^{2} d \|_{L^{2}}+C \|\nabla d\|_{L^{6}}^{2} \|\nabla^{2} d \|_{L^{3}}^{2}\notag\\[2mm]
	&\quad +C\|\nabla d\|_{L^{6}}^{2} \||\nabla d|^{2} \|_{L^{6}} \| \nabla d  \|_{L^{4}}^{2} \notag\\
	&\leq \frac{1}{4}\||\nabla d|| \nabla^{2} d| \|_{L^{2}}^{2}+C \|\nabla d\|_{L^{2}}^{\frac{1}{2}} \|\nabla^{3} d \|_{L^{2}}^{\frac{3}{2}}\|\nabla u\|_{L^{2}}^{2}\notag\\[2mm]
	&\quad  + C\|\nabla^{3} d \|_{L^{2}}\|\nabla^{2} d \|_{L^{2}}^{3} +  C\|\nabla^{2} d \|_{L^{2}}^{4}\|\nabla d\|_{L^{4}}^{4}  \notag\\
	&\leq \frac{1}{4}\big( \|\nabla^{3} d \|_{L^{2}}^{2}+  \||\nabla d||\nabla^{2} d| \|_{L^{2}}^{2}\big)+C\|\nabla d\|_{L^{2}}^{3} \|\nabla^{2} d \|_{L^{2}}^{3} \|\nabla^{3} d \|_{L^{2}}^{2} +C \|\nabla^{2} d \|_{L^{2}}^{6}\notag\\
	&\quad +C(\|\nabla d\|_{L^{2}}^{2}\|\nabla u\|_{L^{2}}^{2}) \|\nabla u\|_{L^{2}}^{6},
\end{align}
where we have used the Gagliardo--Nirenberg inequalities
\begin{align*}
\|\nabla d \|_{L^{\infty}}   \leq C  \| \nabla d \|_{L^{2}}^{\frac{1}{4}} \| \nabla^{3} d \|_{L^{2}}^{\frac{3}{4}},\quad \|\nabla^{2} d \|_{L^{2}}  \leq C  \| \nabla d \|_{L^{2}}^{\frac{1}{2}}  \| \nabla^{3} d \|_{L^{2}}^{\frac{1}{2}}.
\end{align*}
Multiplying \eqref{xxxxx1} by  $4|\nabla d|^{2} \nabla d$ and  integrating the resultant over  $\mathbb{R}^{3}$, we infer that
\begin{align}\label{xxxxx3}
	&\frac{d}{d t} \|\nabla d\|_{L^{4}}^{4} + \int\big(4|\nabla d|^{2} |\nabla^{2} d |^{2}+2 |\nabla(|\nabla d|^{2})|^{2}\big) \, d x \notag\\
	&=  4 \int|\nabla d|^{2} \nabla d \big(-\nabla(u \cdot \nabla d)+\nabla (|\nabla d|^{2} d ) \big) \, d x \notag\\
	&\leq  C \int\big(|\nabla d|^{4}|\nabla u|+|\nabla d|^{3} |\nabla^{2} d ||u|+|\nabla d|^{4} |\nabla^{2} d |+|\nabla d|^{6}\big) \, d x \notag\\
	&\leq \frac{1}{8}  \||\nabla d|  |\nabla^{2} d |\|_{L^{2}}^{2}+C \big(\||\nabla u| |\nabla d|\|_{L^{2}}^{2}+\|\nabla d \|_{L^{6}}^{6}+\||u|  |\nabla^{2} d |\|_{L^{2}}^{2}\big)\notag\\
	&\leq \frac{1}{4}\big( \|\nabla^{3} d \|_{L^{2}}^{2}+  \||\nabla d||\nabla^{2} d| \|_{L^{2}}^{2}\big)+C\|\nabla d\|_{L^{2}}^{3} \|\nabla^{2} d \|_{L^{2}}^{3} \|\nabla^{3} d \|_{L^{2}}^{2} +C(\|\nabla d\|_{L^{2}}^{2}\|\nabla u\|_{L^{2}}^{2}) \|\nabla u\|_{L^{2}}^{6}.
\end{align}

Adding \eqref{xxxxx2} to \eqref{xxxxx3}, we obtain that
\begin{align*}
	& \frac{d}{d t} \big( \|\nabla^{2} d \|_{L^{2}}^{2}+\|\nabla d\|_{L^{4}}^{4}\big)  +  \|\nabla d_{t} \|_{L^{2}}^{2}+ \|\nabla^{3} d \|_{L^{2}}^{2}+\||\nabla d| |\nabla^{2} d |\|_{L^{2}}^{2} \notag\\
	&\leq  C\|\nabla d\|_{L^{2}}^{3} \|\nabla^{2} d \|_{L^{2}}^{3} \|\nabla^{3} d \|_{L^{2}}^{2} +C \|\nabla^{2} d \|_{L^{2}}^{6}  +C(\|\nabla d\|_{L^{2}}^{2}\|\nabla u\|_{L^{2}}^{2}) \|\nabla u\|_{L^{2}}^{6}.
\end{align*}
Integrating the above inequality over $t\in(0,T)$ and choosing $\varepsilon>0$ in \eqref{p1} sufficiently small so that
\begin{align}\label{sm5}
C\sup _{t \in[0, T]}\big(\|\nabla  d(t)\|_{L^{2}}^{2}\|\nabla^{2} d(t)\|_{L^{2}}^{2}\big)^{\frac{3}{2}} \leq \frac{1}{4},
\end{align}
we derive \eqref{xxxxx0}.
\end{proof}

Based on Lemmas \ref{lem3}--\ref{lem5}, we have the following corollary.
\begin{corollary}\label{co4}
Under the hypotheses of Proposition \ref{pro},  it holds that
\begin{align}\label{c9}
	& \sup _{t \in[0, T]}\Big[ \|\nabla u(t)\|_{L^{2}}^{2}+ (\bar{\rho}+1)  \| \sqrt{\rho}\theta(t)\|_{L^{2}}^{2}+\| \nabla^{2} d(t)\|_{L^{2}}^{2}+\|  \nabla d(t)\|_{L^{4}}^{4}\Big] \notag\\
	&\quad\quad +\int_{0}^{T}\left(  \| \sqrt{\rho}\dot{u}\|_{L^{2}}^{2}+  (\bar{\rho}+1)  \| \nabla \theta\|_{L^{2}}^{2}+ \|\nabla d_{t} \|_{L^{2}}^{2}+ \|\nabla^{3} d \|_{L^{2}}^{2}+\||\nabla d|| \nabla^{2} d| \|_{L^{2}}^{2}\right) d t \notag\\[2mm]
	&  \leq C  \left(\|\nabla u_{0}\|_{L^{2}}^{2}+(\bar{\rho}+1) \| \sqrt{\rho_{0}}\theta_{0}\|_{L^{2}}^{2}+\| \nabla^{2} d_{0}\|_{L^{2}}^{2}+\| \nabla d_{0}\|_{L^{4}}^{4}\right)= C S_{0}^{\prime\prime}.
\end{align}
\end{corollary}
\begin{proof}
Combining \eqref{xxx1} and \eqref{xxxxx0}, we have
\begin{align}\label{xxx11}
	& \sup _{t \in[0, T]}\Big[\|\sqrt\rho \theta(t)\|_{L^{2}}^{2}+\|\nabla^{2} d(t)\|_{L^{2}}^{2}+\|\nabla  d(t)\|_{L^{4}}^{4}\Big] +\int_{0}^{T}\big( \|\nabla  d_{t} \|_{L^{2}}^{2}+ \|\nabla^{3} d \|_{L^{2}}^{2}+\||\nabla  d||\nabla^{2} d| \|_{L^{2}}^{2}+\|\nabla \theta\|_{L^{2}}^{2}\big) \, d t \notag\\
	&\leq C\big( \|\sqrt{\rho_{0}} \theta_{0}\|_{L^{2}}^{2}+\|\nabla^{2}  d_{0} \|_{L^{2}}^{2}+ \|\nabla  d_{0} \|_{L^{4}}^{4}\big)+C \sup _{t \in[0, T]}\big(\|\nabla u(t)\|_{L^{2}}^{2}\|\rho(t)\|_{L^{3}}\big) \int_{0}^{T}\|\sqrt{\rho} \dot{u}\|_{L^{2}}^{2}\, d t\notag\\
	&\quad  +C\int_{0}^{T}\|\nabla u\|_{L^{2}}^{2}\|\nabla^{2} d\|_{L^{2}}^{4}\, dt +C \int_{0}^{T}\|\nabla^{2} d \|_{L^{2}}^{3} \|\nabla^{3} d \|_{L^{2}}\, dt+C \int_{0}^{T}\|\nabla^{2} d \|_{L^{2}}^{7}\|\nabla d\|_{L^{2}}\, dt\notag\\
	&\quad +C \sup _{t \in[0, T]}(\|\nabla d(t)\|_{L^{2}}^{2}\|\nabla u(t)\|_{L^{2}}^{2}) \int_{0}^{T}\|\nabla u\|_{L^{2}}^{6}\, d t+C\int_{0}^{T} \|\nabla^{2} d \|_{L^{2}}^{6}\, d t\notag\\
	&\leq C\big( \|\sqrt{\rho_{0}} \theta_{0}\|_{L^{2}}^{2}+\|\nabla^{2}  d_{0} \|_{L^{2}}^{2}+ \|\nabla  d_{0} \|_{L^{4}}^{4}\big)+C \sup _{t \in[0, T]}\big(\|\nabla u(t)\|_{L^{2}}^{2}\|\rho(t)\|_{L^{3}}\big) \int_{0}^{T}\|\sqrt{\rho} \dot{u}\|_{L^{2}}^{2} \, d t\notag\\
	&\quad +  C\sup _{t \in[0, T]}\|\nabla^{2} d(t)\|_{L^{2}}^{2}\int_{0}^{T}\|\nabla u\|_{L^{2}}^{2} \, dt\sup _{t \in[0, T]}\|\nabla^{2} d(t)\|_{L^{2}}^{2}\notag\\
	&\quad +\frac{1}{2}  \int_{0}^{T}\|\nabla^{3} d \|_{L^{2}}^{2} \, dt + C \sup _{t \in[0, T]} \|\nabla^{2} d(t) \|_{L^{2}}^{2} \int_{0}^{T}\|\nabla^{2} d \|_{L^{2}}^{2} \, dt\sup _{t \in[0, T]} \|\nabla^{2} d(t) \|_{L^{2}}^{2}\notag\\
	&\quad+C\sup _{t \in[0, T]}  (\|\nabla^{2} d(t)\|_{L^{2}} \|\nabla d(t)\|_{L^{2}}) \sup _{t \in[0, T]}\|\nabla^{2} d(t)\|_{L^{2}}^{2}\int_{0}^{T}\|\nabla^{2} d\|_{L^{2}}^{2}\, d t\sup _{t \in[0, T]}\|\nabla^{2} d(t)\|_{L^{2}}^{2}\notag\\
	&\quad+C\sup _{t \in[0, T]}(\|\nabla d(t)\|_{L^{2}}^{2}\|\nabla u(t)\|_{L^{2}}^{2}) \sup _{t \in[0, T]}\|\nabla u(t)\|_{L^{2}}^{2}\int_{0}^{T}\|\nabla u\|_{L^{2}}^{2}\, d t\sup _{t \in[0, T]}\|\nabla u(t)\|_{L^{2}}^{2},
\end{align}
where we have used
\begin{align*}
	&C\int_{0}^{T}\|\nabla^{2} d \|_{L^{2}}^{3} \|\nabla^{3} d \|_{L^{2}}\, dt\\
	&\leq C \bigg(\int_{0}^{T}\|\nabla^{2} d \|_{L^{2}}^{6} \, dt\bigg)^{\frac{1}{2}}\bigg(\int_{0}^{T}\|\nabla^{3} d \|_{L^{2}}^{2} \, dt\bigg)^{\frac{1}{2}}\\
	&\leq \frac{1}{2}  \int_{0}^{T}\|\nabla^{3} d \|_{L^{2}}^{2} \, dt + C\int_{0}^{T}\|\nabla^{2} d \|_{L^{2}}^{6} \, dt\\
	&\leq \frac{1}{2}  \int_{0}^{T}\|\nabla^{3} d \|_{L^{2}}^{2} \, dt + C \sup _{t \in[0, T]} \|\nabla^{2} d(t) \|_{L^{2}}^{2} \int_{0}^{T}\|\nabla^{2} d \|_{L^{2}}^{2} \, dt\sup _{t \in[0, T]} \|\nabla^{2} d(t) \|_{L^{2}}^{2}
\end{align*}
and
\begin{align*}	
	&C\int_{0}^{T}\|\nabla^{2} d \|_{L^{2}}^{7}\|\nabla d\|_{L^{2}}\, dt\\
	 &\leq C\sup _{t \in[0, T]}  \big[(\|\nabla^{2} d(t)\|_{L^{2}} \|\nabla d(t)\|_{L^{2}}) \|\nabla^{2} d(t)\|_{L^{2}}^{4}\big] \int_{0}^{T}\|\nabla^{2} d\|_{L^{2}}^{2}\, d t\\
	&\leq C\sup _{t \in[0, T]}  (\|\nabla^{2} d(t)\|_{L^{2}} \|\nabla d(t)\|_{L^{2}}) \sup _{t \in[0, T]}\|\nabla^{2} d(t)\|_{L^{2}}^{2}\int_{0}^{T}\|\nabla^{2} d\|_{L^{2}}^{2}\, d t\sup _{t \in[0, T]}\|\nabla^{2} d(t)\|_{L^{2}}^{2}.
\end{align*}
Choose $\varepsilon>0$ in \eqref{p1} sufficiently small so that
\begin{align}\label{sm11}
	C\max\Big\{ & \sup _{t \in[0, T]} \|\nabla^{2} d(t)\|_{L^{2}}^{2}\int_{0}^{T}\|\nabla u\|_{L^{2}}^{2} \, dt,~ \sup _{t \in[0, T]} \|\nabla^{2} d(t) \|_{L^{2}}^{2} \int_{0}^{T}\|\nabla^{2} d \|_{L^{2}}^{2} \, dt,~\sup _{t \in[0, T]} (\|\nabla  d(t)\|_{L^{2}} \|\nabla^{2}  d(t)\|_{L^{2}}), \notag\\
	& \sup _{t \in[0, T]} \|\nabla u(t)\|_{L^{2}}^{2}\int_{0}^{T}\|\nabla u\|_{L^{2}}^{2}\, d t, ~\sup _{t \in[0, T]} (\|\nabla d(t)\|_{L^{2}}^{2}\|\nabla u(t)\|_{L^{2}}^{2}) \Big\}  \leq \frac{1}{16},
\end{align}
which combined with \eqref{xxx11} implies that
\begin{align}\label{xxx12}
	& \sup _{t \in[0, T]}\Big[\|\sqrt\rho \theta(t)\|_{L^{2}}^{2}+\|\nabla^{2} d(t)\|_{L^{2}}^{2}+\|\nabla  d(t)\|_{L^{4}}^{4}\Big] +\int_{0}^{T}\big( \|\nabla  d_{t} \|_{L^{2}}^{2}+ \|\nabla^{3} d \|_{L^{2}}^{2}+\||\nabla  d||\nabla^{2} d| \|_{L^{2}}^{2}+\|\nabla \theta\|_{L^{2}}^{2}\big) \, d t\notag\\
	&\leq  C\big( \|\sqrt{\rho_{0}} \theta_{0}\|_{L^{2}}^{2}+\|\nabla^{2}  d_{0} \|_{L^{2}}^{2}+ \|\nabla  d_{0} \|_{L^{4}}^{4}\big)+C \sup _{t \in[0, T]}\big(\|\nabla u(t)\|_{L^{2}}^{2}\|\rho(t)\|_{L^{3}}\big) \int_{0}^{T}\|\sqrt{\rho} \dot{u}\|_{L^{2}}^{2} \, d t.
\end{align}
Multiplying \eqref{xxx1} by $\bar{\rho}$ gives that
\begin{align}\label{c5}
	&\bar{\rho}\sup _{t \in[0, T]}\|\sqrt\rho \theta(t)\|_{L^{2}}^{2}+\bar{\rho}\int_{0}^{T}\|\nabla \theta\|_{L^{2}}^{2} \, d t \notag\\
	&\leq C \bar{\rho}\|\sqrt{\rho_{0}}\theta_{0}\|_{L^{2}}^{2}+C\bar{\rho} \sup _{t \in[0, T]}\big(\|\nabla u(t)\|_{L^{2}}^{2}\|\rho(t)\|_{L^{3}}\big) \int_{0}^{T}\|\sqrt{\rho} \dot{u}\|_{L^{2}}^{2}\, d t\notag\\
	&\quad  +C\bar{\rho}\int_{0}^{T}\|\nabla u\|_{L^{2}}^{2}\|\nabla^{2} d\|_{L^{2}}^{4}\, dt +C\bar{\rho} \int_{0}^{T}\|\nabla^{2} d \|_{L^{2}}^{3} \|\nabla^{3} d \|_{L^{2}}\, dt+C\bar{\rho} \int_{0}^{T}\|\nabla^{2} d \|_{L^{2}}^{7}\|\nabla d\|_{L^{2}}\, dt.
\end{align}
Substituting \eqref{xxx12} and \eqref{c5} into \eqref{xxxx0}, we deduce that
\begin{align}\label{c6}
	& \sup _{t \in[0, T]}\Big[\|\nabla u(t)\|_{L^{2}}^{2}+ (\bar{\rho}+1)  \| \sqrt{\rho}\theta(t)\|_{L^{2}}^{2}+\| \nabla^{2} d(t)\|_{L^{2}}^{2}+\|  \nabla d(t)\|_{L^{4}}^{4}\Big] \notag\\
	& \quad +\int_{0}^{T}\left(  \| \sqrt{\rho}\dot{u}\|_{L^{2}}^{2}+  (\bar{\rho}+1)  \| \nabla \theta\|_{L^{2}}^{2}+ \|\nabla d_{t} \|_{L^{2}}^{2}+ \|\nabla^{3} d \|_{L^{2}}^{2}+\||\nabla d|| \nabla^{2} d| \|_{L^{2}}^{2}\right) d t \notag\\[2mm]
	&\leq C \left(\|\nabla u_{0}\|_{L^{2}}^{2}+ (\bar{\rho}+1)  \| \sqrt{\rho_{0}}\theta_{0}\|_{L^{2}}^{2}+\| \nabla^{2} d_{0}\|_{L^{2}}^{2}+\| \nabla d_{0}\|_{L^{4}}^{4}\right)\notag\\[2mm]
	& \quad+C (\bar{\rho}+1)  \sup _{t \in[0, T]}\big(\|\nabla u(t)\|_{L^{2}}^{2}\|\rho(t)\|_{L^{3}}\big) \int_{0}^{T}\|\sqrt{\rho} \dot{u}\|_{L^{2}}^{2} d t\notag\\
	& \quad  +C\bar{\rho}\int_{0}^{T}\|\nabla u\|_{L^{2}}^{2}\|\nabla^{2} d\|_{L^{2}}^{4}\, dt +C\bar{\rho} \int_{0}^{T}\|\nabla^{2} d \|_{L^{2}}^{3} \|\nabla^{3} d \|_{L^{2}}\, dt+C\bar{\rho} \int_{0}^{T}\|\nabla^{2} d \|_{L^{2}}^{7}\|\nabla d\|_{L^{2}}\, dt\notag\\
	&\quad +C \sup _{t \in[0, T]} K_{1}(t)  \int_{0}^{T}\|\nabla \theta\|_{L^{2}}^{2} \, d t+C\int_{0}^{T} K_{2}(t) \, d t\notag\\[2mm]
	&\leq C  \left(\|\nabla u_{0}\|_{L^{2}}^{2}+(\bar{\rho}+1) \| \sqrt{\rho_{0}}\theta_{0}\|_{L^{2}}^{2}+\| \nabla^{2} d_{0}\|_{L^{2}}^{2}+\| \nabla d_{0}\|_{L^{4}}^{4}\right)\notag\\[2mm]
	& \quad +\frac{1}{2}  \int_{0}^{T}\|\nabla^{3} d \|_{L^{2}}^{2}\, dt +C\int_{0}^{T} K_{2}(t) \, d t,
\end{align}
where $\varepsilon>0$ in \eqref{p1} is sufficiently small such that
\begin{align}\label{sm7}
	C\max\Big\{ &(\bar{\rho}+1)  \sup _{t \in[0, T]}(\|\nabla u(t)\|_{L^{2}}^{2}\|\rho(t) \|_{L^{3}}),~~\sup _{t \in[0, T]}(\|\nabla  d(t)\|_{L^{2}} \|\nabla^{2}  d(t)\|_{L^{2}}),\notag\\
	&(\bar{\rho}^{2} +\bar{\rho})\sup _{t \in[0, T]} \|\nabla^{2} d(t) \|_{L^{2}}^{2} \int_{0}^{T}\|\nabla^{2} d \|_{L^{2}}^{2} \, dt\Big\}  \leq \frac{1}{16}.
\end{align}
In addition, noting that
\begin{align*}
	&C\int_{0}^{T}\|\nabla^{2} d\|_{L^{2}}^{\frac{3}{2}}\|\nabla^{3} d\|_{L^{2}}^{\frac{1}{2}}\||\nabla d||\nabla^{2} d|\|_{L^{2}} dt\\
	&\leq C\int_{0}^{T}\|\nabla^{2} d\|_{L^{2}}^{2}\|\nabla^{3} d\|_{L^{2}}^{\frac{3}{2}}\|\nabla  d\|_{L^{2}}^{\frac{1}{2}} \, dt\\
	&\leq \frac{1}{16}\int_{0}^{T}\|\nabla^{3} d\|_{L^{2}}^{2} \, dt+C \int_{0}^{T}\|\nabla d\|_{L^{2}}^{2}\|\nabla^{2} d\|_{L^{2}}^{8}  \, dt\\
	&\leq \frac{1}{16}\int_{0}^{T}\|\nabla^{3} d\|_{L^{2}}^{2} \, dt+C \sup _{t \in[0, T]} (\|\nabla d(t)\|_{L^{2}} \|\nabla^{2} d(t)\|_{L^{2}} )   \int_{0}^{T}\|\nabla d\|_{L^{2}} \|\nabla^{2} d\|_{L^{2}}^{7}  \, dt,
\end{align*}
\begin{align*}
	&C\int_{0}^{T} \|\nabla \theta\|_{L^{2}}\||\nabla d||\nabla^{2} d|\|_{L^{2}}\, dt\\
	&\leq C\int_{0}^{T}\|\nabla \theta\|_{L^{2}}\|\nabla^{2} d\|_{L^{6}} \|\nabla  d\|_{L^{3}} \, dt\\
	&\leq C\int_{0}^{T}\|\nabla \theta\|_{L^{2}}\|\nabla  d\|_{L^{2}}^{\frac{1}{2}}\|\nabla^{2} d\|_{L^{2}}^{\frac{1}{2}} \|\nabla^{3}  d\|_{L^{2}} \, dt\\
	&\leq \frac{1}{16}\int_{0}^{T}\|\nabla \theta\|_{L^{2}}^{2} \, dt+C \sup _{t \in[0, T]} (\|\nabla d(t)\|_{L^{2}} \|\nabla^{2} d(t)\|_{L^{2}}) \int_{0}^{T}  \|\nabla^{3} d\|_{L^{2}}^{2}  \, dt,
\end{align*}
and
\begin{align*}
	&\int_{0}^{T} K_{2}(t)\, dt\\
	&=\int_{0}^{T}\Big(\|\rho\|_{L^{3}}\|\nabla u\|_{L^{2}}^{2}\||\nabla d||\nabla^{2} d|\|_{L^{2}}^{2} + \bar{\rho}\|\nabla \theta\|_{L^{2}}\|\nabla d\|_{L^{2}}^{\frac{1}{2}} \|\nabla^{2} d \|_{L^{2}}^{\frac{1}{2}} \|\nabla d_{t} \|_{L^{2}} +\bar{\rho}\|\nabla  d\|_{L^{2}} \|\nabla^{2} d\|_{L^{2}}\|\nabla  d_{t}\|_{L^{2}}^{2} \\
	&\quad +\bar{\rho}\|\nabla u\|_{L^{2}}^{2}\|\nabla^{2} d\|_{L^{2}}^{4} +\bar{\rho} \|\nabla^{2} d\|_{L^{2}}^{3} \|\nabla^{3} d\|_{L^{2}} +\bar{\rho} \|\nabla^{2} d\|_{L^{2}}^{7} \|\nabla  d\|_{L^{2}}\\
	&\quad+\|\nabla d\|_{L^{2}}^{\frac{1}{2}} \|\nabla^{2} d \|_{L^{2}}^{\frac{1}{2}} \|\nabla d_{t} \|_{L^{2}}\||\nabla d||\nabla^{2} d|\|_{L^{2}}+ \bar{\rho}^{\frac{1}{2}}\| \rho\|_{L^{3}}^{\frac{1}{2}}\|\nabla \theta\|_{L^{2}}\|\nabla u\|_{L^{2}}\||\nabla d||\nabla^{2}d|\|_{L^{2}}   \\
	&\quad +\|\nabla u\|_{L^{2}}\|\nabla^{2}d\|_{L^{2}}^{2}   \||\nabla d||\nabla^{2}  d|\|_{L^{2}}+\|\nabla^{2} d\|_{L^{2}}^{\frac{3}{2}}\|\nabla^{3} d\|_{L^{2}}^{\frac{1}{2}}\||\nabla d||\nabla^{2} d|\|_{L^{2}}  \\
	&\quad +  \|\nabla^{2} d\|_{L^{2}}^{\frac{7}{2}}\|\nabla d\|_{L^{2}}^{\frac{1}{2}}\||\nabla d||\nabla^{2} d|\|_{L^{2}}+\|\nabla \theta\|_{L^{2}}\||\nabla d||\nabla^{2} d|\|_{L^{2}}\Big)\, dt,
\end{align*}
can be absorbed by the left-hand side of \eqref{c6}, provided that $\varepsilon>0$ in \eqref{p1} is taken sufficiently small so that \eqref{sm11}, \eqref{sm7}, and
\begin{align}\label{sm9}
	C\max\Big\{ (\bar{\rho}+1)  \sup_{t\in[0,T]}(\|\nabla^{2} d (t)\|_{L^{2}}\|\nabla d(t)\|_{L^{2}}),~~\sup_{t\in[0,T]}(\|\rho(t) \|_{L^{3}}\|\nabla u(t)\|_{L^{2}}^{2})\Big\}  \leq \frac{1}{16},
\end{align}
are all satisfied. Consequently, \eqref{c9} follows.
\end{proof}

We now present the proof of Proposition \ref{pro}.
\begin{lemma}\label{co1}
Under the hypotheses of Proposition \ref{pro}, it holds that, for any $t \in[0, T]$,
\begin{align}\label{zui}
\mathcal{B}(t) \leq \frac{3}{2} \varepsilon.
\end{align}
\end{lemma}

\begin{proof}
Multiplying \eqref{c9} by $\big(1+\bar{\rho}+\frac{1}{\bar{\rho}}\big)$ and combining the resultant with \eqref{c3}, we obtain that
\begin{align*}
\mathcal{B}(t)=&\bigg(1+\bar{\rho}+\frac{1}{\bar{\rho}}\bigg)S_{t}^{\prime} S_{t}^{\prime\prime}+( \bar{\rho}^{2}+\bar{\rho}) \int_{0}^{T}\|\nabla \theta\|_{L^{2}}^{2} d t \sup _{t \in[0, T]}\|\rho(t)\|_{L^{3}}\notag\\
	&\quad + \sup_{t\in[0,T]}\big[\|\nabla u(t)\|_{L^{2}}^{2}+\big(1+\bar{\rho}+\bar{\rho}^{2}\big)\|\nabla^{2} d(t) \|_{L^{2}}^{2}   \big]  \int_{0}^{T}\big(\|\nabla u\|_{L^{2}}^{2} +\|\nabla^{2} d \|_{L^{2}}^{2}\big)\, d t\notag\\
 \leq& C \bigg(1+\bar{\rho}+\frac{1}{\bar{\rho}}\bigg) S_{0}^{\prime} S_{0}^{\prime\prime}= C \varepsilon_{0}\leq \frac{3}{2} \varepsilon,
\end{align*}
provided that $\varepsilon \geq \tfrac{2}{3}C \varepsilon_{0}\triangleq C_{1}\varepsilon_{0}$, as the desired \eqref{zui}.
\end{proof}

\begin{lemma}\label{co2}
Under the hypotheses of Proposition \ref{pro}, it holds that, for any  $(x, t) \in \mathbb{R}^{3} \times[0, T]$,
\begin{align*}
\rho(x, t) \leq \frac{3}{2} \bar{\rho}.
\end{align*}
\end{lemma}
\begin{proof}
Fix $(x, t) \in \mathbb{R}^{3} \times[0, T]$ and a constant $\delta>0$. Define
\begin{align*}
\rho^{\delta}(y, s)=\rho(y, s)+\delta \exp \left\{-\int_{0}^{s} \operatorname{div} u(X(\tau ; x, t), \tau) d \tau\right\}>0,
\end{align*}
where  $X(s ; x, t)$  is the flow map determined by
\begin{align*}
	\begin{cases}
	\frac{d}{d s} X(s ; x, t)=u(X(s ; x, t), s), ~~0 \leq s<t, \\
	X(t ; x, t)=x.
\end{cases}
\end{align*}
It follows from \eqref{1}$_{1}$ that
\begin{align*}
\frac{d}{d s} \ln \rho^{\delta}(X(s ; x, t), s)=-\operatorname{div} u(X(s ; x, t), s).
\end{align*}
Consequently,
\begin{align}\label{xu1}
Y^{\prime}(s)=g(s)+b^{\prime}(s),
\end{align}
where
\begin{align}\label{s1}
	Y(s)=\ln \rho^{\delta}(X(s ; x, t), s), ~~ g(s)=-\frac{P(X(s ; x, t), s)}{2 \mu+\lambda},~~
		b(s)=-\frac{1}{2 \mu+\lambda} \int_{0}^{s} F(X(\tau ; x, t), \tau)  d \tau.
\end{align}

A direct computation shows that
\begin{align}\label{s2}
F(X(\tau ; x, t), \tau)= & -(-\Delta)^{-1} \operatorname{div}\left[(\rho u)_{\tau}+\operatorname{div}(\rho u \otimes u)\right]-(-\Delta)^{-1} \operatorname{div} \operatorname{div} M(d) \notag\\
= & -\left[(-\Delta)^{-1} \operatorname{div}(\rho u)\right]_{\tau}-u \cdot \nabla(-\Delta)^{-1} \operatorname{div}(\rho u) \notag\\
& +u \cdot \nabla(-\Delta)^{-1} \operatorname{div}(\rho u)-(-\Delta)^{-1} \operatorname{div} \operatorname{div}(\rho u \otimes u) -(-\Delta)^{-1} \operatorname{div} \operatorname{div} M(d)\notag\\
=& -\frac{d}{d \tau}\big[(-\Delta)^{-1} \operatorname{div}(\rho u)\big]+ [u_{i}, R_{i j} ] (\rho u_{j} )-(-\Delta)^{-1} \operatorname{div} \operatorname{div}M(d),
\end{align}
where $[u_i,R_{ij}]=u_iR_{ij}-R_{ij}u_i$ and $R_{ij}=\partial_i(-\Delta)^{-1}\partial_j$ denotes the Riesz transform on $\mathbb{R}^3$.
Therefore, \eqref{s1} and \eqref{s2} imply that
\begin{align}\label{s3}
	b(T)-b(0) \leq  & C\sup _{s \in[0, T]}  \|(-\Delta)^{-1} \operatorname{div}(\rho u)(s) \|_{L^{\infty}} +C\int_{0}^{T} \| [u_{i}, R_{i j} ] (\rho u_{j} )\|_{L^{\infty}} \, d \tau\notag\\
	& +C \int_{0}^{T} \|(-\Delta)^{-1} \operatorname{div} \operatorname{div} M(d)  \|_{L^{\infty}} \, d \tau \triangleq\sum_{i=1}^{3} N_{i}.
\end{align}

Using the Gagliardo--Nirenberg and H\"older's inequalities, the commutator estimates,  \eqref{xu0}, \eqref{c9}, and \eqref{c3}, we derive that
\begin{align}
	N_{1}& \leq C \bar{\rho}^{\frac{3}{4}}\sup_{s\in[0,T]}\bigg(\|\sqrt{\rho} u(s)\|_{L^{2}}^{\frac{1}{2}}\|\nabla u(s)\|_{L^{2}}^{\frac{1}{2}}\bigg),\label{xu3}\\
N_{2}&\leq C\int_{0}^{T} \| [u^{i}, R_{i j} ] (\rho u^{j} ) \|_{L^{\infty}} d \tau\notag\\
	 &\leq C\int_{0}^{T}\|[u, R_{i j}](\rho u)\|_{L^{3}}^{\frac{1}{5}}\|\nabla[u, R_{i j}](\rho u)\|_{L^{4}}^{\frac{4}{5}}d \tau \notag\\
	 &\leq C \int_{0}^{T} \bar{\rho}\|\nabla u\|_{L^{2}}\|\nabla u\|_{L^{6}} d \tau \notag\\
	 &\leq  C \int_{0}^{T} \bar{\rho}\|\nabla u\|_{L^{2}}\bigg(\bar{\rho}^{\frac{1}{2}}\|\sqrt{\rho} \dot{u}\|_{L^{2}}+\bar{\rho}\|\nabla \theta\|_{L^{2}}+\||\nabla d||\nabla^{2} d|\|_{L^{2}}\bigg) d \tau \notag\\
	 &\leq   C \bigg(\bar{\rho}^{2}\int_{0}^{T}\|\nabla u\|_{L^{2}}^{2} d \tau\bigg)^{\frac{1}{2}}\bigg(\int_{0}^{T} \bar{\rho}\|\sqrt{\rho} \dot{u}\|_{L^{2}}^{2} d \tau\bigg)^{\frac{1}{2}} \notag\\
	 &\quad+ C \bigg(\bar{\rho}^{2}\int_{0}^{T}\|\nabla u\|_{L^{2}}^{2} d \tau\bigg)^{\frac{1}{2}}\bigg(\bar{\rho}^{2}\int_{0}^{T}\|\nabla \theta\|_{L^{2}}^{2} d \tau\bigg)^{\frac{1}{2}}+ C  \bigg(\bar{\rho}^{2}\int_{0}^{T}\|\nabla u\|_{L^{2}}^{2} d \tau\bigg)^{\frac{1}{2}}\bigg(\int_{0}^{T}\||\nabla d|| \nabla^{2} d |\|_{L^{2}}^{2} d \tau\bigg)^{\frac{1}{2}}\notag\\
	 &\leq C \bigg[\bigg(1+\bar{\rho}+\frac{1}{\bar{\rho}}\bigg) S_{0}^{\prime} S_{0}^{\prime\prime}\bigg]^{\frac{1}{2}}\label{xu4},\\
	N_{3}&\leq C \int_{0}^{T}  \|(-\Delta)^{-1} \operatorname{div} \operatorname{div} M(d)  \|_{L^{\infty}}  d \tau     \notag\\
	&  \leq C \int_{0}^{T} \|(-\Delta)^{-1} \operatorname{div} \operatorname{div} M(d)  \|_{L^{6}}^{\frac{1}{2}}  \|\nabla (-\Delta)^{-1} \operatorname{div} \operatorname{div} M(d)  \|_{L^{6}}^{\frac{1}{2}}   d \tau\notag\\
	&\leq C \int_{0}^{T}\||\nabla d||\nabla^{2} d|\|_{L^{2}}^{\frac{1}{2}}\||\nabla d||\nabla^{2} d|\|_{L^{6}}^{\frac{1}{2}}  d \tau\notag\\
	&\leq C \int_{0}^{T}\|\nabla d\|_{L^{\infty}}\|\nabla^{2} d\|_{L^{2}}^{\frac{1}{2}}\|\nabla^{3} d\|_{L^{2}}^{\frac{1}{2}}  d \tau\notag\\
	&\leq C \int_{0}^{T} \|\nabla^{2} d\|_{L^{2}} \|\nabla^{3} d\|_{L^{2}}   d \tau\notag\\
	 &\leq C \int_{0}^{T} \bar{\rho}^{\frac{1}{2}} \|\nabla^{2} d\|_{L^{2}}  \bar{\rho}^{-\frac{1}{2}}\|\nabla^{3} d \|_{L^{2}} d \tau\notag\\
	&\leq  C  \bigg(\bar{\rho} \int_{0}^{T}\|\nabla^{2} d\|_{L^{2}}^{2} d \tau\bigg)^{\frac{1}{2}} \bigg(\int_{0}^{T}\bar{\rho}^{-1}\|\nabla^{3} d\|_{L^{2}}^{2} d \tau\bigg)^{\frac{1}{2}}\notag\\
	&\leq C \bigg[\bigg(1+\bar{\rho}+\frac{1}{\bar{\rho}}\bigg) S_{0}^{\prime} S_{0}^{\prime\prime}\bigg]^{\frac{1}{2}}\label{xu5}.
\end{align}
Integrating \eqref{xu1} over  $[0, T]$ and using \eqref{s3}--\eqref{xu5}, we obtain that
\begin{align}\label{s4}
	\ln \rho^{\delta}(x, s) \leq & \ln (\bar{\rho}+\delta)+b(T)-b(0) \notag\\
	\leq & \ln (\bar{\rho}+\delta)+ C \bar{\rho}^{\frac{3}{4}}\bigg(\|\sqrt{\rho} u(s)\|_{L^{2}}^{\frac{1}{2}}\|\nabla u(s)\|_{L^{2}}^{\frac{1}{2}}\bigg)+C \bigg[\bigg(1+\bar{\rho}+\frac{1}{\bar{\rho}}\bigg) S_{0}^{\prime} S_{0}^{\prime\prime}\bigg]^{\frac{1}{2}}\notag\\
	\leq& \ln (\bar{\rho}+\delta) +\ln \frac{3}{2},
\end{align}
provided that  $\varepsilon$ in \eqref{p1} is chosen sufficiently small. Letting $\delta\to0^{+}$, we deduce that
\begin{align*}
\rho \leq \frac{3}{2} \bar{\rho}.
\end{align*}
This completes the proofs of Lemma \ref{co2}.
\end{proof}

Now we are ready to prove Proposition \ref{pro}.
\begin{proof}[Proof of Proposition \ref{pro}]
Proposition \ref{pro} follows from Lemmas \ref{co1} and \ref{co2}.
\end{proof}

\section{Proof of Theorem \ref{thm1}}\label{sec3}
Combining Proposition \ref{pro} with the blow-up criterion \eqref{up} yields the global existence and uniqueness of the strong solution.
Indeed, a uniform upper bound for $\rho$ has been obtained in Lemma \ref{co2}.
Moreover, taking $s_{1}=4$, $r_{1}=6$ and $s_{2}=4$, $r_{2}=6$ in \eqref{up} and using \eqref{c3} and \eqref{c9}, there exists a constant $\bar{C}>0$ independent of $T_{\max}$ such that
\begin{align}\label{zuihou}
	&\sup_{t \in[0, T_{\max}]} \|\rho(t)\|_{L^{\infty}}
	+\int_{0}^{T_{\max}}\|u\|_{L^{6}}^{4}\,dt
	+\int_{0}^{T_{\max}}\|\nabla d\|_{L^{6}}^{4}\,dt \notag\\
	&\le
	\sup_{t \in[0, T_{\max}]}\|\rho(t)\|_{L^{\infty}}
	+C\sup_{t \in[0, T_{\max}]}\|\nabla u(t)\|_{L^{2}}^{2}
	\int_{0}^{T_{\max}}\|\nabla u\|_{L^{2}}^{2}\,dt \notag\\
	&\quad
	+C\sup_{t \in[0, T_{\max}]}\|\nabla^{2} d(t)\|_{L^{2}}^{2}
	\int_{0}^{T_{\max}}\|\nabla^{2} d\|_{L^{2}}^{2}\,dt \notag\\[2mm]
	&\le \bar{C},
\end{align}
which combined with Lemma \ref{local} implies that $T_{\max}=\infty$.

\section*{Conflict of interest}
The authors have no conflicts to disclose.

\section*{Data availability}
No data was used for the research described in the article.

\end{document}